\documentclass [twoside,openright]{amsart}
\usepackage[utf8]{inputenc}
\usepackage[T1]{fontenc}
\usepackage{indentfirst} 

\usepackage{tikz-cd}
\usepackage{amsbsy}
\usepackage{amssymb,latexsym} 
\usepackage{amsmath,amsthm} 
\usepackage{amssymb}
\usepackage{amsthm}
\usepackage{tikz-cd}
\usepackage{graphicx}
\usepackage{dcolumn}
\usepackage{hyperref}
 \usepackage{mathtools}

\usepackage{aliascnt}

\newcolumntype{2}{D{.}{}{2.0}}
\newcommand{\cA}{\mathcal{A}}

\newcommand{\cF}{\mathcal{F}}

\newcommand{\cI}{\mathcal{I}}

\newcommand{\cL}{\mathcal{L}}

\newcommand{\cO}{\mathcal{O}}

 \newcommand{\cQ}{\mathcal{Q}}
  \newcommand{\cK}{\mathcal{K}}

\newcommand{\cPT}{\mathbf{PaT}}

\newcommand{\cPaP}{\mathbf{PaP}}

\newcommand{\fB}{\frak{B}}

\newcommand{\fG}{\mathfrak{I}}

\newcommand{\fN}{\mathfrak{N}}

\newcommand{\fb}{\frak{b}}

\newcommand{\N}{\mathbb{N}}
\newcommand{\R}{\mathbb{R}}
\newcommand{\Q}{\mathbb{Q}}
\newcommand{\Z}{\mathbb{Z}}

\newcommand{\Ss}{\mathbb{S}}
\newcommand{\Pp}{\mathbb{P}}

\newcommand{\bB}{{\mathbf B}}

\newcommand{\bI}{{\mathbf I}}

\newcommand{\bS}{{\mathbf S}}

\newcommand{\bbP}{\mathbb{P}}
\newcommand{\F}{\mathbb{F}}

\usepackage{amsfonts}

\def\bu{\mathbf{u}}

\newcommand{\mPaB}{{\bf mPaB}}
\newcommand{\CoB}{{\bf CoB}}
\newcommand{\mCoB}{\bf{mCoB}}
\newcommand{\PaB}{{\bf PaB}}

\newtheorem*{thm-non}{Theorem}
\newtheorem{thm}{Theorem}
\newtheorem{lem}{Lemma}
\newtheorem{cor}{Corollary}
\newtheorem{prop}{Proposition}

\theoremstyle{definition}
\newtheorem{dfn}{Definition}

\theoremstyle{remark}
\newtheorem{rem}{Remark}
\newtheorem{ex}{Example}
\theoremstyle{remark}

\theoremstyle{plain}

\newtheorem*{thm*}{Theorem}
\numberwithin{thm}{subsection}
\numberwithin{lem}{subsection}
\numberwithin{prop}{subsection}
\numberwithin{cor}{subsection}
\numberwithin{rem}{subsection}
\title[]{Manin conjecture for statistical pre-Frobenius manifolds, hypercube relations and motivic Galois group in coding}
\author{No\'emie C. Combe }
\address{Max Planck Institute for Mathematics in the Sciences,
Inselstr. 22,04103, Leipzig}
\dedicatory{To Yuri Ivanovitch Manin on the occasion of his birthday with admiration and gratefulness.
Thank you so much for teaching me so many incredible things and for opening my mind to a new world.\\
\includegraphics[scale=0.35]{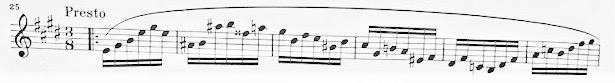} \\ (Paganini, Capriccio 3 for violin solo, {\it Presto})}
\email{noemie.combe@mis.mpg.de}
\thanks{I acknowledge the Minerva Grant from the Max Planck Society for supporting my work.}


\date{\today}
\begin{document}
\maketitle

\begin{abstract}
This article develops, via the perspective of (arithmetic) algebraic geometry and category theory, different aspects of geometry of information. 
First, we describe in the terms of Eilenberg--Moore algebras over a Giry monad, the collection $Cap_n$ of all probability distributions on the measurable space $(\Omega_n, \cA)$ (where $\Omega$ is discrete with $n$ issues) and it turns out that there exists an embedding relation of Segre type among the product of $Cap_n$'s. We unravel hidden symmetries of these type of embeddings and show that there exists a hypercubic  relation. 
Secondly, we show that the Manin conjecture---initially defined concerning the diophantine geometry of Fano varieties---is true in the case of exponential statistical manifolds, defined over a discrete sample space.  
Thirdly, we introduce a modified version of the parenthesised braids ($\mPaB$), which forms a key tool in code-correction. This modified version $\mPaB$ presents all types of mistakes that could occur during a transmission process. We show that the standard parenthesised braids $\PaB$ form a full subcategory of $\mPaB$. We discuss the role of the Grothendieck--Teichm\"uller group in relation to the modified parenthesised braids. Finally, we prove that the motivic Galois group is contained in the automorphism $Aut(\widehat{\mPaB}).$ We conclude by presenting an open question concerning rational points, Commutative Moufang Loops and information geometry.
\end{abstract}

\tableofcontents
\medskip

%ABSTRACT. Any natural language can be considered as a tool for producing large databases (consisting of texts, written, or discursive). This tool for its description in turn requires other large databases (dictionaries, grammars etc.). Nowadays, the notion of database is associated with computer processing and computer mem- ory. However, a natural language resides also in human brains and functions in human communication, from interpersonal to intergenerational one. We discuss in this survey/research paper mathematical, in particular geometric, constructions, which help to bridge these two worlds.

\section{Introduction}
~

The terminology {\it geometry of information} refers to models of databases subject to noise. This connects to quantification, storage, and communication of digital information. Applications of fundamental topics of information theory include source coding, data compression, and channel coding or error detection and correction. In this paper, we consider algebraic structures occurring in geometry of information and we prove surprising connections between the theory of geometry of information and diophantine geometry (see for instance~\cite{Cor} for an introduction).

\smallskip 

We unravel tight bridges between objects of information geometry (such as manifolds of probability distributions and codes) and diophantine geometry and algebraic geometry.
In the first part of this paper, we work with the space of probability distributions on finite sets (see~\cite{Am10} for an introduction and ~\cite{CoCoNen21,CoCoNenFICC,CoMa20} for new developments). Probability distributions are used in many problems such as machine learning, vision, statistical inference, neural networks and others, this development provides a strong tool for many areas. We consider the class of statistical manifolds of exponential type.  The aim of the first part of the paper is to consider the structure of the space of probability distributions and the category of these spaces.  

Our work is subdivided into three main parts. The first part regards the collection of all probability distributions on the measurable (discrete with $n$ issues) space $Cap_n$  and their embeddings into $Cap_m$ (where $m>n$). It turns out that a hidden hypercubic symmetries appear. 
The second part, proves that the Manin conjecture concerning rational points on a Fano variety can be extended to the case of information geometry. The second part, regards a codes/ error-correcting codes aspect of information geometry and we show a tight relation to the motivic Galois group and a modified version of  parenthesised braids, serving as a way of encoding all possible errors occurring a given word. 

\smallskip 

In the first part, we give a proof of the following statement:
\begin{thm-non}(Thm. \ref{T:1})
Let $Cap_n=Cap(\Omega_{n}, \cA_{n})$ be the collection of all probability distributions on the measurable (discrete) space $(\Omega_n, \cA)$ where $\Omega$ is formed from $n$ outcomes.
Then, the diagram of all embeddings of multi-product of $\underbrace{Cap_2\times\cdots\times Cap_2}_{n+1\, times}$ in $Cap(\Omega_{2^{n+1}}, \cA_{2^{n+1}})$ has the structure of an $n$-cube.  
\end{thm-non}
\smallskip 

In the second part of this paper, we show an extension of the Manin conjecture to a wider family of objects, showing thus deep connections 
between information geometry and arithmetics/ algebraic geometry. In particular, we show that Manin's conjecture concerning the diophantine geometry of Fano varieties~\cite{ManConj} holds in the case of exponential statistical manifolds, defined over a discrete sample space. Initially, the Manin conjecture states the following. 

{\it ``Let $V$ be a Fano variety (defined over a number field $K$);
let $H$ be a height function, relative to the anticanonical divisor, and assume that $V(K)$
 is Zariski dense in $V$. Then there exists a non-empty Zariski open subset 
$U\subset V$ such that the counting function of $K$-rational points of bounded height, defined by the set
$N_{{U,H}}(B)=\#\{x\in U(K):H(x)\leq B\}$
for $B\geq 1$, satisfies the relation $N_{{U,H}}(B)\sim cB(\log B)^{{\rho -1}}$,
as $B\to \infty$, where $c$ is a constant.''}
A {\it reformulation} of this statement in terms of the pre-Frobenius statistical manifolds of exponential type (defined on a discrete sample space) is given.  We show that this conjecture is {\it true} for those manifolds of information geometry. 

\begin{thm-non}(Thm. \ref{T:2})
Consider $S=\{p(q,\theta)\}$ an exponential statistical manifold (over a discrete sample space $(\Omega,\cA)$ of finite dimension. 
\begin{itemize}
\item Let $T=(\Q^*)^m$ be the $\Q$-torus of the exponential statistical manifold given by the probability coordinates. 
\item Consider $\Pp_\Sigma$ the smooth $\Q$-compactification of the torus $T$ i.e. a smooth, projective $\Q$-variety in which $T$ lies as a dense open set and $\Sigma$ is a Galois invariant regular complete fan. 
\item Let $k$ be the rank of the Picard group $Pic(\Pp_\Sigma)$. 
\end{itemize}
Then, there is only a finite number $N(T,\cK^{-1},B)$ of $\Q$-rational points $x \in T(\Q)$ having the anticanonical height $H_{\cK^{-1}}(x)\leq B$. 
Moreover, as $B\to \infty$:
\[N(T,\cK^{-1},B)=\frac{ \Theta(\Sigma)}{(k-1)!}\cdot B(logB)^{k-1}(1+o(1)),\]
where $\Theta(\Sigma)$ is a constant.
\end{thm-non}

In former works of Manin and collaborators~\cite{Mouf,QuOp}, was shown the existence of Moufang patterns encoding various symmetries appearing naturally in models, related to storing and transmitting information such as information spaces. By Moufang patterns we have in mind in particular loops (such as Moufang loops). The latter form non-associative analogs of groups. For the case of spaces of probability distributions on finite sets the symmetries of these spaces have the structure of Commutative Moufang Loops.  

Loops and quasigroups turn out to play a central role when it comes to considering geometry of information. 
The aspect relating geometry of information and (virtually) non-commutative Moufang Loops appears in the context of error-correcting codes and algebraic-geometry codes~\cite{Mouf}.

In the second part of this work, we consider an aspect of geometry of information directly related to semantics and to the theory of error-correcting codes and to errors~\cite{Err,sem}. Any natural language can be considered as a tool for producing large databases. Communication (or transmission of information) refers to the process by which a sender communicates a message (i.e. a union of sequences of letters forming words defined in a given alphabet and associated to a given meaning) to a receiver. During a given communication the message can arrive distorted. We investigate the cases where the message is subject to distortion (or is coded) and arrives finally modified.  

A modification can take different aspects such as: permutations of letters, replacement of a letter by another one, removal of a letter or on the contrary extension of words by adding letters and sequences of words. In particular, when it comes to coding, latin squares can be used to decode the message. 

\smallskip 

We consider the space of all possible modifications of a words (indexed by their length) and suppose that our words are parenthesised. Suppose that we have a pair of parenthesised words $w$ and $w'$ of the same length. It turns out that an object which is perfect for the description of this situation and which also offers a geometric vision of paths of errors made during a transmission is tightly related to the groupoid of parenthesised braids, introduced to study the Grothendieck--Teichm\"uller group and which was first considered by Drinfeld.

In order to cover all sorts of error transmissions we introduce an {\it enriched version of the groupoid of parenthesised braids}, denoted $\mPaB$ for {\bf m}odified {\bf pa}renthesised {\bf b}raids. This modification of the classical parenthesised braid is necessary due to the fact that the words are allowed to have repeating letters and this is strongly related to loops and quasigroups. The braids are consequently impacted. Thus we equip the standard braids with two supplementary operations called {\it pinching} and {\it attaching} operations. 

This groupoid of modified parenthesised braids inherits naturally the operations of cabling $d_i$, strand removal $s_i$, extension $d_0$ as well as $\square$, the coproduct functor defined by setting each individual parenthesised braid $B$ to be group-like, i.e. $\square(B)=B\otimes B$ and $\sigma$  the elementary braid on two strands.

\smallskip 

It turns out that the there is a strong connection between the space of all possible transmission errors and arithmetics. Indeed the pro-unipotent Grothendieck--Teichm\"uller group (and thus the motivic Galois group) are  included in the automorphism  of the pro-unipotent completion of $\widehat{\mPaB}$. 

\begin{thm-non}(Thm. \ref{T:GT} and Cor. \ref{C:mot})
The motivic Galois group is contained in the automorphism of the pro-unipotent completion of the modified  parenthesised braids $Aut(\widehat{\mPaB})$.
\end{thm-non}

\smallskip 

To conclude, we discuss an open question relating the Commutative Moufang Loops (CML) structure arising in the symmetries of the spaces of probability distributions on a discrete set. The appearance of the simplest CML’s in the algebraic-geometric setup is motivated by smooth cubic curves in a projective plane $\Pp_K^2$ over a field $K$. It is in particular shown that the set $E$ of $K$-points of such a curve $X$ forms a CML. Regarding the result of the first part of the paper, we conjecture that the set $E$ of $\Q$-points of a pre-Frobenius statistical manifold  forms also a CML.
 
\smallskip 
\subsection*{Plan of the paper}

-Sec.\ref{S:1}. of the paper is devoted to considering $Cap_n$ which is the collection of all probability distributions on the measurable space $(\Omega_n, \cA)$. The sample space is discrete and $\Omega_n$ has $n$ outputs. There exists an associated monad (called the Giry monad) and an algebra over it (Eilenberg--Moore algebra).  We discuss in particular the relation among the product of $Cap_n$, which turns out to be hypercubic. 
Secondly, we prove that the the Manin conjecture holds for (pre-Frobenius) the statistical manifolds related to exponential families and defined over a discrete (finite) set (Sec.\ref{S:ManinConj}). 

-Sec.  \ref{S:3} we discuss another aspect that geometry of information can take, via codes and error codes. It serves as an intermezzo between Sec. 1 and the Sec. 3. and prepares the ground for what follows. After recalling definitions on Loops and quasigroups we study in particular the algebraic properties of the space of modified words and show that quasigroups and loops offer a perfect set up for this. 

-Sec. \ref{S:mPaB} we introduce our modified parenthesised braids, which forms a key tool in code-correction. We show that the standard parenthesised braids $\PaB$ are a full subcategory of $\mPaB$. We discuss the role of the Grothendieck--Teichm\"uller group in relation to the modified parenthesised braids (Sec.\ref{S:GT}). Finally, 
we end the section by showing that the motivic Galois group is contained in the automorphism $Aut(\widehat{\mPaB}).$ We conclude finally by presenting an open question concerning rational points, Commutative Moufang Loops and information geometry \ref{S:conj}.

\medskip 

\section{Statistical pre-Frobenius manifolds in relation to algebraic geometry and Manin's conjecture}\label{S:1}
\subsection{Categorical introduction of considered objects}
%%%%%%%%
Dealing with classical information theory leads to working in the following framework.  Let $(\Omega,\cA,P_0)$ a probability space where $\Omega$ is the space of elementary outputs, $\cA$ is the $\sigma$-algebra of events and $P_0$ is a probability measure, (usually  $P_0\ll \mu$, $P_0[A]=\int_A f_0d\mu$ for some ($\sigma$-finite) measure $\mu$).

\smallskip
 
-- The algebra $\fB_c$ is  the  (commutative) algebra of all bounded measurable functions $f$ on the space of elementary outcomes $\Omega$.  

$ \fB_c $ algebra with respect of addition and multiplication by a scalar of function $f:\omega\to \R,\quad f^{-1}(x)\in \cA, \quad x\in \R$

\smallskip

-- The probability state of an object is determined by a nonnegative, normalized, normal (i.e., ultra- weakly continuous, or what is the same, monotone continuous) linear functional $\Phi_c$ one, $\fB_c$,  $\Phi_c:\fB_c\to \R$, which is the expectation with respect to some probability $P_0$:
\[ \Phi_c(f) =\mathbb{E}_{P_0}[f].\]

\vspace{5pt}
-- The set $\fG(\fB_c)$ of all states $\Phi_c$ of an object is a convex closed set in the pre-dual space $\fB_c$,\, $(\fB_c)^\star=\fB_c$.

\vspace{5pt}
-- The idempotents of $\fB_c$ are just the indicators (characteristic functions) of the measurable sets (elements of $\cA$), these subspaces are called events (or ``yes-no'' experiments).

Before we enter a categorical definition, let us mention that the collection of all probability measures on a measurable space $(\Omega,\cA)$ of elementary outcomes
is a {\it convex subset} of the semi-ordered linear space of measures of bounded variations on  $(\Omega,\cA)$. 
In some cases, it useful to remark that the collection of all probability measures on $(\Omega,\cA)$ is equipped with a norm, giving rise to a metric space. However, this aspect will not be important to us here. 

\smallskip 

These measures are endowed with the supplementary property that they are {\it invariant} under maps of the collection of probability measures induced by invertible measurable maps of the sample space $(\Omega,\cA)$. This means that given a pair of sample spaces $(\Omega_1,\cA_1)$ and $(\Omega_2,\cA_2)$ equipped with their corresponding collection of probability measures, say $\{P_{\theta}^i, \, \theta\in \Theta \}$ where $i\in {1,2}$ and with (same) parameter set $\Theta\in \mathbb{R}^n$ one can develop a notion of {\it equivalence}: $\{P_{\theta}^1, \, \theta\in \Theta \}$ and $\{P_{\theta}^2, \, \theta\in \Theta \}$ are said to be equivalent whenever there exist Markov maps $\Pi^{12}$ and $\Pi^{21}$ such that $P^{1}_{\theta}\Pi^{12}= P^{2}_{\theta}$ and  $P^{2}_{\theta}\Pi^{21}=P^{1}_{\theta}$, for  any $\theta\in\Theta$.

We call $Cap$ the collection of all probability distributions on the measurable space $(\Omega,\cA).$
The discussion above leads to defining a category denoted $CAP$, where objects are isomorphic classes of collections of all probability distributions on  the measurable spaces $(\Omega,\cA)$; morphisms are given by the Markov maps. These Markov maps correspond to statistical decision rules in the sense of Wald. 

Further algebraic operations on the objects of the category are allowed and are defined as follows. One can define a {\it direct product} of measurable spaces. This construction implies the existence of a tensor product on the collection of all probability measures on those measurable spaces. This multiplication is functorial with respect to the Markov category.

To give an example, let us take $Cap_2$, where $\Omega=\{\omega_1,\omega_2\}$ and the probability distributions $p=\langle p_1,p_2\rangle.$
Defining $Cap_2\times Cap_2$ is given by $\{\omega_1,\omega_2\} \times\{\omega'_1,\omega'_2\}$ and this corresponds to $\{\omega_1\omega'_1,\omega_1\omega'_2,\omega'_1\omega_1,\omega'_2\omega_1\}$, which can be rewritten as: 
$\{\omega^{''}_1,\omega^{''}_2,\omega^{''}_3,\omega^{''}_4\}$. If $p=\langle p_1,p_2\rangle$ and $q=\langle q_1,q_2\rangle$ are the corresponding probability distributions then the tensor product on the probability distributions is such that: 
\[p\otimes q=\langle p_1q_1,p_1q_2, p_2q_1,p_2q_2\rangle.\]

We shall investigate more precisely what happens during this multiplication process in $Cap_n$. In particular we show, using Segre type of embeddings that we have a hypercube type of relation within these operations.

However, note that in this paper, we are not interested in the quantum aspect of information theory and we limit ourselves to the case where the algebra is commutative. Hence, we do not consider the quantum information geometry aspect which requires a von Neumann algebra $\fb$ of bounded linear operators acting on Hilbert space.  

This algebra $\fb$ corresponds to a (generally non-commutative) analogue of the classical commutative algebra of all bounded measurable functions on the space of elementary outcomes. The Hermitian elements of the algebra $\fb$ are called bounded observables. The probability state of an object is determined by a (nonnegative, normalised, normal, monotone continuous) linear functional $\phi$ on the algebra $\fb$. The set of all states $S(\fb)$  of a given object forms a convex closed set in the pre-dual space $\fb_*$, where $(\fb_*)^*=\fb$. 

Moreover, analogous constructions concerning the Markov maps can be defined and the system of all Markov maps of all collections $S(\fb)$ forms an algebraic category.

To summarise, there exist tight relations between convex sets and the sets of states in the probabilistic computation (discrete or continuous) and in quantum computation. We will explore this from a more categorically aspect. In particular, we invoke the Giry monad structure and define an Eilenberg--Moore algebra over this monad to consider our convex sets for information geometry.

\smallskip 

Consider the category ${\rm Set}$ of finite sets. Given an object $X$ of   ${\rm Set}$, we define 
\[\Delta_X=\big\{\tilde{f}: X\to [0,1]\, |\, \tilde{f}\, \text{has finite support and} \, \sum_{x\in X} \tilde{f}(x)=1\]
the simplex over $X$. It is the set of formal finite convex combinations of elements from $X$. 
Elements of $\Delta_X$ are the discrete probability distributions over $X$. The mapping from the set $X$ to the set $\Delta_X$ can be made functorial (known as the simplex functor) and defined such that given a morphism of sets $f:X\to Y$ one defines $\Delta(f):\Delta_X\to \Delta_Y$. 
One may define, for these convex sets, the algebraic structure of a monad $(\Delta,\mu,\eta)$, where the unit is given by $\eta:X\to \Delta_X$ and the multiplication is defined by $\mu:\Delta_X^2\to \Delta_X$. This monad is commutative. Moreover, given an object $X\in Ob({\rm Set})$ and the {\it structure map} $\gamma:\Delta_X\to X$ commutativity for the following diagrams is satisfied:
\begin{center} \begin{tikzcd}
  X \arrow[rd,"\eta_X"] \arrow[r,"Id_X"] & X \\
  & \Delta_X\arrow[u, "\gamma"]
  \end{tikzcd}
 \begin{tikzcd}
 \Delta_{\Delta_X} \arrow[d,"\mu_X"] \arrow[r, "\Delta_{\gamma}"] & \Delta_X \arrow[d,"\gamma"] \\
  \Delta_X \arrow[r, "\gamma"]& X.
  \end{tikzcd}
\end{center}

Taking a pair $(X,\gamma)$ allows to work in an Eilenberg--Moore algebra for the distribution monad $(\Delta,\mu,\eta)$ over the symmetric monoidal category of sets. An algebra morphism $f:(X,\gamma)\to (X',\gamma')$  is a continuous map such that the following diagram commutes:
\begin{center}
  \begin{tikzcd}
 \Delta_{X} \arrow[d,"\Delta(f)"] \arrow[r, "{\gamma}"] & X \arrow[d,"f"] \\
  \Delta_{X'} \arrow[r, "\gamma'"]& X'.
  \end{tikzcd}
\end{center}

An Eilenberg--Moore algebra of the monad $(\Delta,\mu,\eta)$ is a map of the form $\gamma: \Delta_X \to X$ s.t. $\gamma \circ \eta=id$ and $\gamma\circ \mu=\gamma \circ \Delta_\gamma.$ Note that each category of algebras for a monad on sets is cocomplete.  
Regarding the category of Eilenberg--Moore algebras it is both complete and cocomplete.

A more global idea hides behind this, in the context of probability distributions and concerning the relation between algebras for the Giry monad~\cite{Giry} and the convex spaces formed by the collection of probability distributions. This comes from the following statement:

The category of algebras for the Giry monad is isomorphic to the category of $G$-partitions. Here by $G$-partition we mean that for $X$ (in a fully general setting $X$ is a polish space, i.e. a separable metric space for which a complete metric exists) corresponds to a collection $\{G(x)\, |\, x\in X\}$, which forms a positive convex partition for $\Delta$ into closed sets indexed by $X$.
Moreover, $\delta_x\in G(x)$ (where $\delta_x$ is the Dirac measure on $x$) holds for all $x\in X$, and the set valued map $x\mapsto G(x)$ is $k$-upper-semicontinuous.

Now, algebras over a commutative monad admit a tensor product. By $\star$ we denote the monoidal multiplication of $\Delta$.  
So, given $\Delta$-algebras $(A,a)$ and $(B,b)$ their tensor product is the object $A\otimes_\Delta B$ given by:
\[\Delta(\Delta A\otimes \Delta B)\xrightarrow{\star}\Delta \Delta(A\otimes B)\xrightarrow{\mu}\Delta(A\otimes B)\to A\otimes_\Delta B.\]

\smallskip 
Now that we have explicitly shown the algebraic structure of the convex spaces of probability distributions, and discussed the tensor product operation for algebras over the Giry monad, we are interested in proving the existence of hidden symmetries appearing within the multiplication relations between $Cap_n$'s. In particular, this leads to proving the existence of a hypercube relation.

Let us go back to the previous example $Cap_2\times Cap_2\hookrightarrow Cap_4$. Recall the tensor product relation on the probability distributions ``\`a la Morozova--Chentsov''~\cite{MoCh}, where one considers the probability distributions under the shape of a vector in an affine space. Regarding our example, this gives us: 
\[p\otimes q=\langle p_0q_0,p_0q_1, p_1q_0,p_1q_1\rangle,\]
so that we can consider $p\otimes q$ as the 4-tuple: $\langle p_0q_0,p_0q_1, p_1q_0,p_1q_1\rangle=\langle p'_0,p'_1,p'_2,p'_3\rangle$ defined in the affine space. This relation can obviously be generalised for any $n$ i.e for $Cap_n$. 

Now, from an affine $n$-tuple $p=\langle p_1,\cdots, p_n\rangle$ one can take easily the homogeneous coordinates: $P=[p_1:p_2:\cdots: p_n]$. So, now whenever we consider the product $P\otimes Q$, where we consider it from the projective perspective, 
one has $P=[p_0:p_1:\cdots: p_n]$ and $Q=[q_0:q_1:\cdots: q_m]$. 
This leads to defining a (real) Segre embedding, which looks as follows:
\[h:\mathbb{P}^n\times \mathbb{P}^m \hookrightarrow\mathbb{P}^{(n+1)(m+1)-1},\]
where $\mathbb{P}^n$ is the {\it real} projective space. 

Taking a pair of points $(P,Q)$ in the projective space $\mathbb{P}^n\times \mathbb{P}^m$ (corresponding to elements in $Cap$) one can define the product:
\[([p_0:p_1:\cdots p_n],[q_0:q_1:\cdots q_m])\mapsto [p_0q_0:p_0q_1:\cdots:p_iq_j:\cdots p_nq_m].\]

In particular, going back to the $Cap_2$ case, with $P=[p_0:p_1]$ (resp. $Q=[q_0:q_1]$), where $0\leq p_i\leq 1$ and $p_0+p_1=1$  (resp. $0\leq q_i\leq 1$ and $q_0+q_1=1$) we have the following commutative diagram:

\begin{center}
  \begin{tikzcd}
\mathbb{P}^1\times \mathbb{P}^1 \times \mathbb{P}^1 \arrow[d,"Id\times h_{(23)}"] \arrow[r,"h_{(12)}\times Id"] & \mathbb{P}^3  \times \mathbb{P}^1 \arrow[d,"h_{((12)3)}"] \\
 \mathbb{P}^1  \times \mathbb{P}^3 \arrow[r,"h_{1(23)}"] & \mathbb{P}^7
  \end{tikzcd}
\end{center}

The following embedding 
\[\underbrace{\mathbb{P}^1\times \cdots \times\mathbb{P}^1}_{n\ times}\to \mathbb{P}^{2^n-1}\]
corresponds to a generalised Segre embedding. This allows to consider the relation between $Cap_2$ and $Cap_{2^n}$ and leads to the following remark on the geometry of $Cap_{2^n}$.
\begin{rem}
Note that this highlights a new way to show the existence of a paracomplex structure on objects in $Cap$. Note that a paracomplex projective space of dimension $n$ is identified to a product of projective spaces of the same dimension $n$ (and defined over the real numbers) i.e.  $\mathbb{P}^n\times \mathbb{P}^n$. Since, we can use the Segre embedding and the construction above we can see that for instance in $Cap_4$ we have an embedded paracomplex projective space.  Using the generalised Segre embedding, the statement generalises for $Cap_{2^n}$. The paracomplex structure has been mentioned in the work~\cite{CoMa20}. However, this gives another approach to that result.  
 \end{rem}
The generalised Segre embedding implies the existence of $(n-1)$-hypercube relations when considering $\underbrace{Cap_2\times \cdots \times Cap_2}_{n\, times}.$ 
Let us start with the following relation. 
\[\underbrace{\mathbb{P}^1\times \cdots \times\mathbb{P}^1}_{n\ times}\to \mathbb{P}^{2^n-1}.\]
We proceed by analogy on $Cap_n$ so that we can in fact obtain a similar relation as in the Segre emebedding: 

\[\underbrace{Cap_2\times \cdots \times Cap_2}_{n\, times}\to Cap_{2^n}\]

Take $n=3$. For the product $Cap_2\times Cap_2\times Cap_2$ we have that the following square commutative diagram:
\begin{center}
  \begin{tikzcd}
\mathbb{P}^1\times \mathbb{P}^1 \times \mathbb{P}^1 \arrow[rr, "h_{(12)}\times Id_3"] \arrow[d,"Id_1\times h_{23}" ] & &
\mathbb{P}^3  \times \mathbb{P}^1 \arrow[d, "h_{(12) 3}" ] \\
 \mathbb{P}^1  \times \mathbb{P}^3 \arrow[rr, "h_{1 (23)}"] && \mathbb{P}^7
  \end{tikzcd}
\end{center}
\begin{thm}\label{T:1}
The diagram of embeddings of $\underbrace{Cap_2\times \cdots \times Cap_2}_{n+1}$ in $Cap_{2^{n+1}}$ has the structure of an $n$-cube.
\end{thm}

\begin{proof}
The proof can be done by using a bijection between the set of vertices and edges of the generalised Segre embedding diagram and the set of vertices/ edges constructing the $n$-cube.

Take a product $\underbrace{\mathbb{P}^1\times\cdots \times \mathbb{P}^1}_{n+1}$. For any pair of (adjacent) projective spaces in this cartesian product, to which the Segre embedding is applied, add a pair of parenthesis (one open and one closed) such that
 \[(\mathbb{P}^1\times\mathbb{P}^1\times \cdots \underbrace{(\mathbb{P}^1\times\mathbb{P}^1)}_{i,i+1}\cdots \times \mathbb{P}^1)\hookrightarrow (\mathbb{P}^1\times\mathbb{P}^1\times \cdots \underbrace{\mathbb{P}^3}_{i}\cdots \times \mathbb{P}^1).\] 

The construction goes as follows and it goes by induction on $n\geq1$. 
\begin{itemize}
\item To any  parenthesised product $\underbrace{(\mathbb{P}^1\times \mathbb{P}^1) \times \mathbb{P}^1\times \cdots\times (\mathbb{P}^1\times \mathbb{P}^1)}_{n+1}$ corresponds a binary word $x=(x_1,\cdots, x_n)$ with $n$ letters such that $x_i\in\{0,1\}$, such that $\underbrace{\mathbb{P}^1\overbrace{\times}^{(x_1,}\mathbb{P}^1\overbrace{\times}^{x_2,} \mathbb{P}^1\cdots\overbrace{\times}^{\cdots,x_n)}\mathbb{P}^1}_{n+1}$ where:
\end{itemize}
\[\begin{cases}
 x_i=0 & \text{if there exists {\bf no} parenthesis for the pair} (i,i+1) \text{of projective spaces:}\\
 & \underbrace{\mathbb{P}^1\times\cdots \overbrace{\mathbb{P}^1\times \mathbb{P}^1}^{i,i+1}\times \mathbb{P}^1}_{n+1}. \\
 x_j=1 & \text{if there exists a parenthesis for the pair} \underbrace{\mathbb{P}^1\times\cdots \overbrace{(\mathbb{P}^1\times \mathbb{P}^1)}^{j,j+1}\times \mathbb{P}^1}_{n+1}\\
\end{cases}\]
So the number of parenthesis is given by the number of units in the word. For each unit added, the remaining zeros of the word correspond to the remaining projective spaces which have not yet been paired.
\begin{itemize}
\item for any $n\geq1$, the product $\underbrace{\mathbb{P}^1\times\cdots \times \mathbb{P}^1}_{n+1}$ corresponds to  $\underbrace{(0,\cdots ,0)}_{n}$.
\item $(1,1,\cdots,1)$ corresponds to the projective space $\mathbb{P}^{2^{n+1}-1}$. 
\item Each combination of parenthesis being encoded by a binary word $(x_1,\cdots, x_n)$ corresponds to the vertex of the Segre embedding diagram.
\item Add one parenthesis to a given combination. This corresponds to adding a unit to the word i.e. $x_j=1$ if we have $\mathbb{P}^1\times\mathbb{P}^1\times \cdots \underbrace{(\mathbb{P}^1\times\mathbb{P}^1)}_{j,j+1}\cdots \times \mathbb{P}^1$ at the $j$-th and $(j+1)$-th position. 
\item An edge of the diagram is drawn whenever to words $x$ and $x'$ differ only by one letter i.e. there exists one unique $j$ such that $x_j\neq x'_j$. For $i\neq j$ we have $x_i=x'_j$.  
\end{itemize}

$\star$ Let us discuss the low dimensional case. Take $\mathbb{P}^1\times\mathbb{P}^1$. The corresponding word has one letter $(x_1)$. This corresponds to $Cap_2\times Cap_2$. Then, by the Segre embedding  $(\mathbb{P}^1\times\mathbb{P}^1)\hookrightarrow \mathbb{P}^3$.  
The new vertex obtained by the Segre embedding modifies the word 0 into the word  1. The diagram is just a segment (so a cube of diemsion 1).

$\star$ We have $\mathbb{P}^1\times \mathbb{P}^1\times \mathbb{P}^1$. Let us use our construction, where vertices of the embedding diagram are encoded by the binary words of length 2, i.e. $(x_1,x_2)$ where $x_i\in \{0,1\}.$
The initial vertex is encoded by $\mathbb{P}^1\times \mathbb{P}^1\times \mathbb{P}^1$ which corresponds to the word $(0,0)$.
Two vertices are connected by an edge whenever the pair of corresponding words differ by only one character. So, for instance, the vertex (0,0) is directly connected by an edge to the vertices $(0,1)$ and $(1,0)$ but not connected to the vertex word (1,1). The diagram is a square. 

$\star$  For the case $Cap_2\times Cap_2\times Cap_2\times Cap_2$, it is easy to check that one has a cube diagram relation.

$\star$ In full generality, the relations of embeddings of $\underbrace{\mathbb{P}^1\times \cdots \times\mathbb{P}^1}_{n+1\ times}$  in the generalised Segre embedding have the structure of a $n$-hypercube graph.  

In fact, this statement follows from the definition of a hypercube (or $n$-cube) which is a graph of order $2^n$, whose vertices are represented by $n$-tuples $(x_1,\cdots x_n)$ where $x_i\in \{0,1\}$ and whose edges connect vertices which differ in exactly one term. 
We use the construction above, where we have established a bijection between the set of vertices indexed by binary words and the parenthesised product of projective spaces; edges of the $n$-cube correspond to applying one Segre embedding to a pair of  parenthesised projective spaces.

So, to conclude we have a hypercube graph relation illustrating the diagram of relations in $Cap_n$.
\end{proof}

We illustrate a four dimensional cube (a tesseract) in the figure below (Fig. 1), where vertices are indexed by words of length 4 and letters are in $\{0,1\}$.
Using the above construction, we can exactly illustrate the Segre embedding relations for
:\[\underbrace{Cap_2\times \cdots \times Cap_2}_{5\, times} \to Cap_{32},\] which in a projective version corresponds to illustrating $\underbrace{\mathbb{P}^1\times \cdots \times \mathbb{P}^1}_{5\, times} \to \mathbb{P}^{31}.$

\begin{center} 
     \begin{tikzpicture}[scale=1.0]
  \draw[fill] (-0, 5) circle (.07cm);  \node (-0,5) at (-0,5.2) {$\scriptstyle (1111)$};
  \node (-0,5) at (-0,5.6) {$\scriptstyle \bbP^{31}$};
  \draw[fill] (-0, 1.9) circle (.07cm); \node (-0,19) at (-0,2.2) {$\scriptstyle (0110)$};
  \draw[fill] (-0, -2) circle (.07cm);  \node (-0,-2) at (-0,-2.4) {$\scriptstyle (1001)$};
  \draw[fill] (-0, -5) circle (.07cm);  \node (-0,-5) at (-0,-5.2) {$\scriptstyle (0000)$};
   \node (-0,-5) at (-0,-5.6) {$\scriptstyle \bbP^1\times\bbP^1\times\bbP^1\times\bbP^1\times\bbP^1$};
 %%%%%%%%%%%%%%%%%%%%%%%%%%%%%    
 \draw[fill] (-4, 3.5) circle (.07cm);  \node (-4,3.5) at (-4.7,3.5) {$\scriptstyle (0111)$};
 \draw[fill] (4, 3.5) circle (.07cm);  \node (4,3.5) at (4.6,3.5) {$\scriptstyle (1110)$};
 %%%%%%%%%%%%%%%%%%%%%%%%%%%
 \draw[fill] (-4, -3.5) circle (.07cm);  \node (-4,5) at (-4.5,-3.5) {$\scriptstyle (0001)$};
 \draw[fill] (4, -3.5) circle (.07cm);  \node (4,-3.5) at (4.5,-3.5) {$\scriptstyle (1000)$};
 %%%%%%%%%%%%%%%%%%%%%%%%%%%%%
  \draw[fill] (-5.7, 0) circle (.07cm);  \node (-5.7,0) at (-6.4,0) {$\scriptstyle (0011)$};
 \draw[fill] (5.7, 0) circle (.07cm);  \node (5.7,0) at (6.2,0) {$\scriptstyle (1100)$};
%%%%%%%%%%%%%%%%%%%%%%%%%%%%%%
  \draw[fill] (-2, 0) circle (.07cm);  \node (-2,0) at (-2.5,0) {$\scriptstyle (0101)$};
 \draw[fill] (2, 0) circle (.07cm);  \node (1.9,0) at (2.5,0) {$\scriptstyle (1010)$};
%%%%%%%%%%%%%%%%%%%%%%%%%%%%%
  \draw[fill] (-1.5, 1.5) circle (.07cm);  \node (-1.5,1.5) at (-1.99,1.6) {$\scriptstyle (1011)$};
\draw[fill] (1.6, -1.5) circle (.07cm);  \node (1.6,-1.5) at (2,-1.8) {$\scriptstyle (0100)$};
%%%%%%%%%%%%%%%%%%%%%%%%%%%%%%
%%%%%%%%%%%%%%%%%%%%%%%%%%%%%
\draw[fill] (-1.6, -1.45) circle (.07cm);  \node (-1.6,-1.345) at (-2.1,-1.6) {$\scriptstyle (0010)$};
\draw[fill] (1.55, 1.35) circle (.07cm);  \node (1.55,1.35) at (2,1.5) {$\scriptstyle (1101)$};
%%%%%%%%%%%%%%%%%%%%%%%%%%%%%%

               \draw[black] (-0,5) -- (-4,3.5);
               \draw[black] (-0,5) -- (4,3.5);   
                   \draw[black] (-4,3.5) -- (-0,1.9);
                    \draw[black] (-0,1.9) -- (4,3.5);
      %%%%%%%%              
                    \draw[black] (-0,-5) -- (-4,-3.5);
               \draw[black] (-0,-5) -- (4,-3.5);    
                   \draw[black] (-4,-3.5) -- (-0,-2);
                    \draw[black] (-0,-2) -- (4,-3.5); 
     %%%%%%%%%%
        \draw[black] (-5.7,0) -- (-4,-3.5);
          \draw[black] (5.7,0) -- (4,3.5);
         \draw[black] (5.7,0) -- (4,-3.5);
         \draw[black] (0,5) -- (-1.5,1.5);
          %%%%%%%%%%
        \draw[black] (2,0) -- (4,3.5);
         \draw[black] (2,0) -- (4,-3.5);
          \draw[black] (-5.7,0) -- (-4,3.5);
         \draw[black] (-2,0) -- (-4,-3.5);
         \draw[black] (-2,0) -- (-4,3.5);
         \draw[black] (-1.5,1.5) -- (-5.7,0);
          \draw[black] (1.6,-1.5) -- (5.7,0);
           \draw[black] (-1.5,1.5) -- (-0,-2);
           \draw[black] (1.6,-1.5) -- (-0,1.9);
                \draw[black] (0,-5) -- (-1.6,-1.45);
               \draw[black] (-5.7,0) -- (-1.6,-1.45);
                \draw[black] (1.9,0) -- (-1.6,-1.45);
                 \draw[black] (0,5) -- (1.55, 1.35);
                 \draw[black] (5.7,0) -- (1.55, 1.35);
                  \draw[black] (0,-5) -- (1.6, -1.45);
                   \draw[black] (0,-2) -- (1.55, 1.35);
                    \draw[black] (-2,0) -- (1.6, -1.5);
                    \draw[black] (-2,0) -- (1.55, 1.35);
                  \draw[black] (0,1.9) -- (-1.6, -1.45);
                  \draw[black] (1.9,0) -- (-1.6, -1.45);
\draw[black] (1.9,0) -- (-1.55, 1.5);

                  \end{tikzpicture}
                  \smallskip 
                  
            Figure 1:  Segre diagram for $\mathbb{P}^1 \times \mathbb{P}^1\times \mathbb{P}^1\times \mathbb{P}^1\times \mathbb{P}^1$ with vertices labeled by binary words (commas have been omitted for simplicity).
                  %\end{figure}

\end{center}

\subsection{Manin's conjecture theorem for discrete exponential families}\label{S:ManinConj}
The theory of exponential varieties reveals the existence of a surprisingly strong connection to diophantine geometry. We show that the asymptotic formula conjectured by Manin, for the case of Fano varieties, concerning the number of $K$-rational points of bounded height with respect to the anticanonical line bundle {\it holds} in the case of a smooth projectivisation of an exponential variety (defined for a discrete finite sample space). This statement extends to the framework of information geometry the conjectured by Manin, which initially was stated in the context of algebraic geometry.

Regarding the previous subsection, we are now working on an object of the category $Cap$. 
Our statement goes as follows: 

\begin{thm}\label{T:2}
Consider $S=\{p(q,\theta)\}$ an exponential statistical manifold (over a discrete sample space $(\Omega,\cA)$) of finite dimension. 
\begin{itemize}
\item Let $T=(\Q^*)^m$ be the $\Q$-torus of the exponential statistical manifold given by the probability coordinates. 
\item Consider $\Pp_\Sigma$ the smooth $\Q$-compactification of the torus $T$ i.e. a smooth, projective $\Q$-variety in which $T$ lies as a dense open set and $\Sigma$ is a Galois invariant regular complete fan. 
\item Let $k$ be the rank of the Picard group $Pic(\Pp_\Sigma)$. 
\end{itemize}
Then, there is only a finite number $N(T,\cK^{-1},B)$ of $\Q$-rational points $x \in T(\Q)$ having the anticanonical height $H_{\cK^{-1}}(x)\leq B$. 
Moreover, as $B\to \infty$:
\[N(T,\cK^{-1},B)=\frac{ \Theta(\Sigma)}{(k-1)!}\cdot B(logB)^{k-1}(1+o(1)),\]
where $\Theta(\Sigma)$ is a constant.
\end{thm}

\begin{rem}
The exponential statistical manifold is a pre-Frobenius manifold and we will can refer to it as the pre-Frobenius statistical manifold (for a definition of pre-Frobenius manifold we refer from instance to~\cite{Man99}).
\end{rem}
The proof of this statement is done in two parts. The first is to state explicitly the relation from exponential varieties (defined as above for finite, discrete sample space) to toric varieties. The second part is to apply the theorem of Batyrev--Tschinkel in this context. 

\subsection{Exponential statistical manifolds for discrete sample space}

A statistical variety (or manifold) can be considered as the  parametrized family of probability distributions $S=\{p(x;\theta)\}$, where $p(x;\theta) =\frac{dP_\theta}{d\mu}$ is the Radon--Nikodym derivative of $P_\theta$ w.r.t. the $\sigma$-finite measure $\mu$ (and it is positive $\mu$-almost everywhere). 
It comes equipped with the following ingredients:
\begin{itemize}
\item the canonical parameters: $\theta= (\theta^1,\dots, \theta^n)\in \R^n$;
\item the symbol $x$ referring to a family of random variables $\{x_i\}$ on a sample space $\Omega$;  
\item $p(x;\theta)$ is the probability distribution parametrized by $\theta$.
\end{itemize}
  
A family $S=\{p(x;\theta)\}$ of distributions is an {\it exponential family} if the density functions can be written in the following way:   
  \[p(x;\theta)= \exp(\theta^ix_i-\Psi(\theta)),\] where   
  \begin{itemize}
  \item $\Psi(\theta)$ is a potential function, which is given by $\Psi(\theta)=\log\int_{\Omega}\exp\{\theta^ix_i\}d\mu$. 
  \item the parameter $\theta$ and  $x=(x_i)_{i\in I}$ (where $I$ is a finite set) have been chosen adequately;
\item the  canonical parameter satisfies $\Theta:=\{\theta\in \R^d: \Psi(\theta)<\infty\}$.
\end{itemize}

 Whenever $p(x;\theta)$ is smooth enough in $\theta$ one can include in the statistical model a structure of an $n$-dimensional manifold. We use the construction of the family $A$ as a manifold, using the atlas $\{U_i,\phi_i\}_{i\in I}$.

From now on suppose that the sample space is finite and discrete, i.e. $\Omega=\{ \omega_1,\cdots,\omega_m\}$. 
A small change of notation is required for practical reasons. This leads us to consider the exponential family of probability distribution defined by:
\begin{equation}\label{E:exp1}
p(q;\theta) = p_0(\omega) \exp \{\theta^i q_i ( \omega) - \Psi ( \theta) \},\, \text{where}\quad p_{0}(\omega)> 0
\end{equation}
and with canonical distribution parameter $\theta=(\theta^1,\cdots,\theta^n) \in \R^n$. Furthermore, we have $\omega \in \Omega$ elements of the sample space and $q_i :\Omega \to \R$ are a family $q=\{q_i\}$ of random variables; $\Psi (\theta)$ is a cumulant generating function.
The  $q_i(\omega), i \in I $ ($I$ is some list of indices) is a function defining directions of the coordinate axes, called statistics (or {\it directional sufficient statistics}). 

\begin{prop}\label{P:tor}
The exponential statistical manifolds (defined as above and over a discrete and finite sample space) have the structure of a real toric variety.
\end{prop} We present the construction below. 

\smallskip 
\begin{proof}
$\bullet$ Let us define $\cQ=\{\bf{q}_1,\cdots,\bf{q}_n\}$, where ${\bf q_i}=(q_{i1},\cdots,q_{im})^T \subset \Z^m$ and the components satisfy $q_{ij}:=q_i(\omega_j)$, with $\omega_j \in \Omega$. The matrix $\cQ$ has size $m\times n$ and its components are integers. Columns are given by the set $\{{\bf q_1},\cdots, {\bf q_n}\}$. The matrix $\cQ=[q_{ij}]_{i=1,\cdots, n; \, j=1,\cdots,m}$, where $q_{ij} = q_i(\omega_j)$ and $q_0(\omega_j)=1$ give the {\it directional statistics.}

\smallskip 

$\bullet$ Put $t_i=e^{\theta^i}\in \R_>$. The monomial is then $t_i^{q_i(\omega_j)}=\exp{\{\theta_iq_i(\omega_j)\}}$ and one can write the following equation:
\[
\exp\left\{\sum_{i=1}^n \theta^iq_{ij}\right\}=\prod_{i=1}^n t_i^{q_{ij}}=\boldsymbol{\tau_j}
\]

So, to conclude, we have that $p(q;t)$ can be rewritten as the product $t_1^{{\bf q}_1}\cdots t_n^{{\bf q}_n}$. 

Moreover, since we assumed that $q_{ij}$ are integers, $\boldsymbol{\tau_j}$ (for $j=1,\cdots,m$) form Laurent polynomials in $t_i$. 
Therefore, each vector ${\bf q_i}$ is identified to a monomial $t^{\bf q_i}$ in the Laurent polynomial ring $\Q[t^{\pm}]$, where $\Q[t^{\pm}]:=\Q[t_1,\cdots,t_m,t_1^{-1},\cdots,t_m^{-1}]$.

Statistically speaking, the monomial $t_i^{q_i(\omega_j)}$ can be interpreted as the probability of having the canonical parameter $\theta_i$ in the direction of $q_i(\omega_j)$ for the event $\omega_j.$ Whereas, the $m$-tuple $(t_1,\cdots,t_m)\in(\Q^*)^m$, where $t_i=\exp{\theta^i}$ form the {\it probability coordinates}.

\smallskip 

Going back to the classical construction of the toric ideal, we apply the following. 
Take the (semigroup) homomorphism:
 \[\pi: \N^n\to \Z^m,\quad \bu=(u_1,\cdots,u_n)\mapsto\sum_{i=1}^n u_i{\bf q_i}.\] 
 The image of $\pi$ is the semigroup:
 \[\N\cQ=\{\lambda_1{\bf q_1}+\cdots \lambda_n{\bf q_n}\, :\, \lambda_1,\cdots, \lambda_n\in \N\}.\]

This map $\pi$ lifts to a homomorphism of semigroup algebras:
\[\hat{\pi}: \Q[{\bf y}]\to \Q[t^{\pm1}],\quad y_i\mapsto t^{\bf q_i},\]
where $\Q[{\bf y}]$ is a polynomial ring in the variables ${\bf y}:=(y_1,\cdots, y_n)$.

It is the kernel of the homomorphism $\hat{\pi}$ that generates the toric ideal $\cI_{T}$ of $\cQ$. The multiplicative group $(\Q^*)^m$ is known as the $m$-dimensional algebraic torus. The variety of the form $V(\cI_{T})$ is the affine toric variety. So, we have shown the existence of an $m$-dimensional algebraic torus for the exponential statistical manifolds. This algebraic torus is given by the {\it probability coordinates} $(t_1,\cdots,t_m)\in(\Q^*)^m$, where $t_i=\exp{\theta^i}$.

Note that for $dim(\cQ)=m$, one can visualise the dense torus using the fact that the set $V(\cI_{T})\cap (\overline{\Q}^*)^m$ is an algebraic group under coordinate-wise multiplication which is isomorphic to the $m$-dimensional torus $T=(\overline{\Q}^*)^m$.

To each point $P\in U_i$, where $(U_i,\phi_i)$ is a chart, we apply the homomorphism construction above. The coordinate functions $y_i$ on the chart $U_i$ can be expressed as Laurent monomials in the adequate coordinates. In changing from one chart to another the coordinate transformation remains monomial. So, this forms a smooth toric variety, where we have a collection of charts $y_i: U_i\to\Q^n$, such that on the intersections of $U_i$ with $U_j$ the coordinates $y_i$ must be Laurent monomials in $y_j$.

A toric variety with a collection of charts determines a system of cones $\{\sigma_a \}$ in $\R^n$. Putting coordinates $x_1,\cdots x_n$ on a given fixed chart $U_0$ the coordinate functions $x^{(a)}$ on the remaining charts $U_a$ can be represented as Laurent monomials in $x_1,\cdots x_n$. 

Furthermore, if we have a regular function $f$ on $U_a$, then it can be represented as a Laurent polynomial in $x_1,\cdots, x_n$.  The regularity condition for a function $f$ on the chart $U_a$ can be expressed in terms of the support of the corresponding Laurent polynomial $\tilde{f}.$ For $\tilde{f}=\sum_{m\in \Z}c_mx^m$ the support of $\tilde{f}$ is the set $\{m\in \Z^n \, |\, c_m\neq 0\}$ and with each chart $U_a$, we associate a cone $\sigma_a$ generated by the exponents of $x_1^{(a)},\cdots, x_n^{(a)}$ as Laurent polynomials in $x_1,\cdots x_n$.

An arbitrary Laurent polynomial $\tilde{f}$ is regarded as a rational function on $X$. Regularity of this function on the chart $U_a$ is equivalent to $supp(\tilde{f})\subset \sigma_a$. Thus, various questions on the rational function $\tilde{f}$ on the toric variety $X$ reduces to the combinatorics of the positioning of $supp(\tilde{f})$ with respect to the system of cones $\{\sigma_a \}$.  

Reciprocally, one can construct a toric variety by specifying a system of cones  $\{\sigma_a \}$ satisfying certain properties. These requirements can be most conveniently stated in terms of the system of dual cones and leads to the notion of fan.
\end{proof}
\begin{cor} 
Consider the exponential statistical variety defined for a discrete finite sample space.    
If we have a regular function $f$ on $U_a$, then it can be represented as a Laurent polynomial in $x_1,\cdots, x_n$.  The regularity condition for a function $f$ on the chart $U_a$ can be expressed in terms of the support of the corresponding Laurent polynomial $\tilde{f}$ and in the exponential variety it is given by the directional statistics. In particular, the cone is generated by the directional statistics. 
\end{cor}

\begin{proof}
This follows from the discussion above. 
\end{proof}

Now, we argue that the Manin conjecture holds for exponential statistical manifolds. Indeed, following the construction of Batyrev--Tschinkel~\cite{BaT}, the Manin conjecture is true for toric varieties. From the above statement (Prop. \ref{P:tor}) it follows that the exponential statistical manifolds (defined over finite sample space) have the structure of a (real) toric variety. Therefore, the conclusion follows easily that a smooth projectivised version of the exponential statistical manifolds defined over finite and discrete sample space satisfies the Manin conjecture.

\smallskip

\section{Words, codes and algebraic structures in information transmission}\label{S:3}
\subsection{Motivation}
As was shown in previous works, Moufang loops and quasigroups are central in information geometry.
We focus on the situation of coding or of error making during a transmission of a given information. It turns out that the algebraic structures of loops and quasigroups offers the right language and formalism to deal with this type of problems. This is starting to be developed in Sec. 2.3 and the following sections. We recall below some results relating structures codes and non-necessarily commutative Moufang loops and quaisgroups.

Commutative Moufang Loops appear in the symmetries of the space of probabilities: {\it automorphisms of order two} that are boundary limits of the reflections of geodesics about the center, come equipped with a structure of a quasigroup. These {\it automorphisms} define a composition law on the set of points that forms an {\it abelian quasigroup}. 

Similarly, non-necessarily commutative Moufang loops and quasigroups appear among the other aspect of information geometry, regrouping around codes/ structure codes (see~\cite{Mouf,QuOp}). We will mention a few results in relation to this in what follows.  

\smallskip 

Based on the works in \cite{Err}, family of codes are defined as follows. We choose and fix an integer $q \geq 2$ and a finite set: the alphabet $A$ of cardinality $q$.  An (unstructured) {\it code $C$} is defined as a nonempty subset $C \subset S$ of words of length $n\geq 1$. 
The sequence $w=(\alpha_i)$ of elements of $A$, where $i = 1,2,...,n$ is called a word $w$ of length $n$. We denote by $n(C)$ the common length of all words in $C$. Such a subset $C$ comes equipped with its code point datum. This is given by a pair $P_C=({\rm R}(C), \delta(C))$, where 
${\rm R}(C)$ is called the {\it transmission rate} and $\delta(C)$ is the {\it relative minimal distance} of the code. 

\begin{itemize}
\item The relative minimal distance of the code $\delta(C)$ is given by the quotient 
$\delta(C):=\frac{d(C)}{n(C)},$ where
 $d(C)=\min\{d(a,b)\, |\, a,b \in C, a\neq b\}$ is the minimal distance between two different words in $C$; $n(C) := n$ and $d(a,b)$ is the Hamming distance between two words: 
$d((\alpha_i),(\alpha'_i)):= card\{i \in (1,\cdots, n)\,  | \, \alpha_i \neq \alpha'_i\};$
\item The transmission rate  ${\rm R}(C)$ depends on the $log_{q}(Card(C))$ i.e. we have: ${\rm R}(C)=\frac{[log_{q}(Card(C))]}{n(C)}$.
\end{itemize}
Note that for our investigations the code point $P_C=({\rm R}(C), \delta(C))$ will not be directly considered, although it is implicitly present. 

As mentioned in earlier works of \cite{Mouf} (section 5.2), Moufang symmetries generally become visible in the so-called {\it structured codes.}  The most studied structure codes appear in linear codes and algebraic-geometric codes. Concerning the former (linear codes) one considers the alphabet $S:= \F_q,$ corresponding to generators of a finite field of cardinality $q$, and $C \subset \F_{nq}$ form $\F_q$-linear subspaces. Concerning the latter (algebraic-geometric codes) one has the same class of alphabets, but the difference is that one considers
$\F_q$-points in an affine (or projective) $\F_q$-scheme with a chosen coordinate system. 

Moufang symmetries appear indirectly in this geometric setting. Their existence can be seen using various and different formalisms, motivated for instance by theoretical physics. Let us recall some definitions on loops and quasigroups. 

\subsection{Quasigroups and Moufang symmetries}

For the convenience of the reader, we recall below the algebraic structures of quasigroups, loops, Moufang loops. 

\begin{enumerate}
\item Let $A$ be a finite set of cardinality $q$. A {\it binary operation} on a set $A$ is a mapping $\diamond:A\times A\to A$ which associates to every ordered pair $(a,b)$ of elements in $A$ a unique element $a\diamond b$. A set with a binary operation is called a {\it magma}. 

\item A quasigroup is a magma (i.e. a set $A$ with a binary multiplication denoted by $\diamond$) such that in the equation $x\diamond y=z$ the knowledge of any two of $x, y, z$ specifies uniquely the third. Latin squares form the multiplication tables of quasigroups. 

\item Based upon the set $A$ a Latin square is a $|A| \times |A|$ array in which each element of $A$ occurs exactly once in each row and exactly once in each column. In particular, for all ordered pairs $(a,b)\in A^2$ there exist unique solutions $x,y \in A$ to the equations: %\break
 $x\diamond a=b,\quad a\diamond y=b$, and those solutions are precisely given by the Latin squares. 
\,
Differently speaking, for each element $r$ of a magma $(A, \diamond)$ one can define the left multiplication: 

\[L(r)=L_x(r):A\to A, x\mapsto r \diamond x\]

and the right multiplication:

\[R(r)=R_x(r):A\to A, x\mapsto x\diamond r.\]

The operators $L$ and $R$ form bijections of the underlying set $A$. We call them left (resp. right) {\it translation} maps.  
In particular, this allows to reformulate the definition of the quasigroup, using the translation maps so that a magma is a combinatorial quasigroup iff the left multiplication $L(r)$ and the right multiplication $R(r)$ are bijective for each element $r$ of $A$.

We can add to this structure the possibility of having a unit denoted $e$ i.e. such that $e\diamond a = a\diamond e = a$ holds for any element $a\in S$. 

\item A quasigroup $(A, \diamond)$ is a nonempty set $A$ equipped with a binary multiplication $\diamond: A \times A \to A$
 and such that, for each $a \in S$, the right and left translation maps $R(a): A \to A$ and $L(a): A \to A$ given by
$R(a) =r\diamond a$ and $L(a) =a\diamond r$, are permutations of $A$. 
If there is a two-sided identity element $1_A = 1_{(A,\diamond)}$ then $A$ is a loop.

\item A  loop is a called Moufang if it is a unital quasigroup (it has a unit and every element is invertible) with a near associativity relation: 
\[(a\diamond b)\diamond (c\diamond d)=a\diamond ((b\diamond c) d),\]
\[a\diamond (a \diamond b) = (a \diamond a) \diamond b,\quad  (a \diamond b) \circ (a \diamond c) = (a \diamond a) \diamond (b \diamond c).\]
where $(a,b,c,d)\in S^4$.
\end{enumerate}

\smallskip 

Going back to our previous discussion on code loops, the loop $\cL$ is, roughly speaking, given by the sequence:
\[0\to R\to \cL \to C\to 0,\]
where $R$ is a ring (which will be more precisely defined below); $C\subset \mathbb{F}^n_{2^r}$ is a linear code, equipped with an additional structure which is introduced in the next paragraph: the ``almost-symplectic structure''. In order to give a flavour to the reader we recall this notion and expose how the loops appear in more details. For further information we refer to~\cite{sem}.

\smallskip 

An almost symplectic structure on a finite dimensional vector space $V$ over $\mathbb{F}_q$ ($q$ odd) is a non-degenerate skew-symmetric form $\omega: V\times V \to \mathbb{F}_q$ where $\omega$ satisfies the
anti-symmetry $\omega(u,v) = -\omega(v,u)$, with $\omega(u,0) = \omega(0,u) = 0$, 
and for any non-null element $u$ in $V$ there exists some $v\in V$ satisfying $\omega(u,v)\neq0$.
A polarisation of the almost-symplectic form is a function $\beta : V \times V\to R$ satisfying the relation
\[\beta(u,v)-\beta(v,u)=\omega(u,v).\]

\smallskip 

Consider the finite field $\mathbb{F}_{2^r}$ and identify it to the residue field: $\cO_K/{\bf m}_K$,  
where $K$ is an unramified extension of degree $r$ of ${\bf Q}_2$; $\cO_K$ is the ring of integers and ${\bf m}_K$ the maximal ideal. 
The ring $R$ in the above short exact sequence is given by $R=\cO_K/{\bf m}^2_K$.

The construction of the almost-symplectic code loop $\cL(V, \beta)$ over $\mathbb{F}_q$ where $q=2^r$ is an extension given by the short exact sequence:
\[0\to R\to \cL(V,\beta) \to V\to 0,\]
where $(V,\beta)$ is an almost-symplectic vector space $(V,\omega)$ with polarization $\beta$ over $\mathbb{F}^n_{2^r}$.

\smallskip 

This setup motivates our investigations concerning codes and error-codes. In particular, we give a construction allowing to take into account all possible errors (or error corrections) occurring during the transmission of some information. The framework of quasigroups and loops fits adequately this type of problem. 

\subsection{Words, codes and algebraic structures}
The algebraic structure of spaces of words and codes are interesting to study. As soon as one associates to words of the code $C$ some given meaning, a code $C$ forms a type of dictionary for a given language. A finite combination of code words form sentences in this language. However, it can happen that given an information encoded by such a sentence it might be distorted during the transmission and so mistakes may appear in the receivers message, changing thus its meaning.

The types of mistakes that can possibly occur are listed below:
\begin{enumerate}
\item letters in a word can be permuted, 
\item one letter can be replaced by another letter (we say that this letter has been translated or shifted to another one),
\item new letters can be added to the word,
\item letters can be lost in the word,
\item new words can be added to the preexisting word. 
\end{enumerate}

In the following part of this section we consider the first two types of mistakes. We argue that quasigroups and loops offer the perfect setting to define these types of operations.
Mistakes of type (3), (4), (5) are considered in the next section where the notion of modified parenthesised braids is introduced.  

\smallskip 
 \begin{ex}
 Consider the alphabet $A_4=\{a,b,c,d\}$ and suppose the Latin square associated is as follows.
 \begin{center}
\renewcommand\arraystretch{1.3}
\setlength\doublerulesep{0pt}
\begin{tabular}{r||*{4}{2|}}
$\cdot$ & a & b & c & d  \\
\hline\hline
$a$ & b & c & d & a \\ 
\hline
 $b$ & c & d & a & b \\ 
\hline
$ c$ & d & a & b & c \\ 
\hline
 $d$ & a & b &c & d \\ 
\hline
\end{tabular}
\end{center}
Then the word $w=(c(ab)d)$ can be distorted using the translation maps as $(L_a(c)(aR_b(b))d)$ and the receiver reads $(d(ad)d)$. 
 \end{ex}
We do not assume commutativity (unless it is clearly stated) i.e. the word {\it bac} is not equivalent to {\it cab}. When it comes to   parenthesised words, associativity is not allowed either, so that $b(ac)$ is not equivalent to $(ba)c$.

\smallskip 

We now introduce the following notations and explicit the corresponding notions. 
\begin{itemize}
\item Consider an alphabet $A$ (finite set of cardinality $n$).
\item By $M_{p}(A)$ we denote the  parenthesised $p$-words formed from the alphabet $A$. Repetitions of letters are allowed. 
\item $M_n(A)$: sum of $M_p(A)\times M_{n-p}(A)$, where $1\leq p\leq n-1$. A word $w\in M_n(A)$ can be written as the concatenation of two smaller words, strictly contained in between an open and a closed parenthesis i.e. we have $w=(w')\circ (w'')$ where $w'$ is of length $p$ and $w''$ of length $n-p$. We call those subwords the {\it blocks} of $w$. 

\item The sum of the family $(M_n(A))_{n\geq 1}$ is denoted $M(A)$.
\item  $M(A)$: the free magma, with composition law $w,w' \mapsto w\circ w'$. 
\item $\mathbb{M}_k(A)$ is used only for  parenthesised $k$-words with {\it distinct letters}. Note that for this notation to be consistent it is necessary that $k\leq n$. In particular, $\mathbb{M}_n(A)$ are the  parenthesised $n$-words with $n$ distinct letters. 
\end{itemize}
There is a clear separation of $M_{n}(A)$ into two subclasses made of those words with distinct letters $\mathbb{M}_n(A)$ and those words with repeating letters.  
\begin{rem}
Concerning the last class of parenthesised words, if we take for example $A=\{a,b,c,d\}$, then an element of $\mathbb{M}_{5}A$ forms a word where letters repeat: the expression $(a(bd))(ba)$ represents an element of $\mathbb{M}_{5}A$. 
\end{rem}

\smallskip 

We now consider the connection between the free magma structure $M(A)$ and the magma $(A,\circ)$ on which the quasigroup acts $GQ=(A, \diamond, L,R)$. Suppose for simplicity that $A$ has cardinality $n$.
Then we have an action on the sequence of letters forming a word $w=(x_1\cdots x_n)$ that we can write as an $n$-tuple i.e. $(x_1,\cdots, x_n)$ such that each entry (letter) of the n-tuple is translated by a left $L$ or right map $R$. To avoid any source of confusion we denote the translation of all letters of the word as $T$. So, we have the following:   
\[QG\times M_n(A)\to M_n(A)\]
 \[(T,  (x_1,\cdots, x_n))\to (T(x_1),\cdots, T( x_n))\]

Recall the construction of $M(A)$. For any $w\in M(A)$, there exists a unique $n\geq 1$ such that $w\in M_n(A).$
For a given pair of words $(w,w')\in  M_p(A)\times M_{q}(A)$ of length $p$ and $q$ respectively we can define a product $w\circ w'$ forming an element of $M_{p+q}(A)$. The set $M(A)$ with the law composition $w,w' \mapsto w\circ w'$ is the free magma. So, we can state the following lemma.

\begin{lem}\label{L:bki}
 Let $A_n$ be a finite set of cardinality $n$. Consider the magma $\fN:=(A_n,\circ)$ and suppose that there exists a quasigroup $QG_n=(A_n, \diamond, R, L)$ acting on the letters of the words of length $n$. 
Then, there exists a unique morphism $g:M(A)\to \fN$ from the free magma on $A_n$ to $\fN$.
\end{lem}

\begin{proof}
A quasigroup is a magma $\fN$ where every element is invertible. Let us define the following bijection $f: A \to \fN$. 
By induction we can construct the morphism $g$ as follows. 
\begin{itemize}
\item Let $f_1=f:M_1(A)\to \fN$, where $M_1(A)=A$. 
\item For $n\geq 2$, we have $f_n:M_n(A)\to \fN$ and given $w\circ w' \in M_p(A)\times M_{n-p}(A)$, 
$f_n(w\circ w')=f_p(w)\circ f_{n-p}(w')$.
\end{itemize}
There exists only one unique morphism $g$ inducing $f_n$ on $M_n(A)$ for all $n\geq 1$. So, $g$ is the unique morphism of $M(A)$ into $\fN$ which extends $f$. 
\end{proof}

\begin{lem}
Consider the quasigroup $(A,\diamond,L,R)$ acting on $M_n(A)$, $n\geq 2$. Then, any permutation of the letters of a word  $w\in M_n(A)$ can be recovered by an adequate combination and choice of translation maps $L$ and $R$.
\end{lem}
\begin{proof}
Any permutation can be obtained by a product of transpositions. Now, the operators $L_x$ and $R_y$ define a transposition iff 
\begin{equation}\label{E:equation}x\diamond a=b\quad \text{and}\quad b\diamond y=a,\end{equation}
where $x,y \in A$. 
So, given a quasigroup $(A,\diamond,L,R)$ any permutation of the letters of a word  $w\in M_n(A)$ can be obtained for every $n\geq 2$ from $L$ and $R$.
\end{proof}

Restricting our attention to $\mathbb{M}_{n}A$, we have the following free magma structure, defined as follows:

\begin{itemize}
\item $\mathbb{M}_{0}A=\emptyset$ 
\item $\mathbb{M}_{1}A=A$ 
\item $\mathbb{M}_{n}A=\sqcup_{p+q=n}\mathbb{M}_{p}A\times \mathbb{M}_{q}A.$ 
\end{itemize}

\begin{rem}
We can interpret $\mathbb{M}_{n}A$ differently as a set of rooted binary planar trees (each vertex has exactly two incoming edges) with $n$ leaves labelled by elements of $A$.
\end{rem}

Let $M = \{M(n)\}$ be the symmetric sequence where $M(n)$ is the subset of $\mathbb{M}_n \{1,\cdots, n\}$ consisting of the monomials in $\{1, ..., n\}$ where each element of the set occurs exactly once. The symmetric group $\Ss_n$ acts from the right on $M(n)$ by permuting the elements of the set $\{1, ..., n\}$. The symmetric sequence $M$ becomes an operad with operadic composition given by replacing letters by monomials (or grafting binary trees). The operad $M$ is called the magma operad.

\begin{cor}
Let $A_n$ be a set of cardinality $n$.
Let $M = \{M(n)\}$ be the symmetric sequence, where $M(n)$ is the subset of $\mathbb{M}_n \{A_n\}$. Then, any permuted $n$-sequence of $M(n)$ in $(A_n)^n$ can be recovered from the action of the quasigroup $(A_n,\diamond,L,R)$ on the elements of $\mathbb{M}_n \{A_n\}$ where for any  transposition $(x_ix_j)$ we put the condition that $L(x_i)=x'_i$ and $R(x_j)=x'_j$ for $x'_i=x_j$ (resp. $x'_j=x_i$) and the rest of the letters remain unchanged.

\end{cor}

\begin{proof}
Every $M(n)$ is formed by all words of length $n$, where letters are all distinct. 
As was previously shown any permutation of a pair of letters in a word $w \in M(n)$ is obtained from $L$ and $R$ by applying condition~\refeq{E:equation}. So, any symmetric sequence in $M(n)$ is obtained by taking a word with distinct letters and  one can apply the operators $L$ and $R$ (and condition~\eqref{E:equation}) to any pairs of letters so as this defines a transposition. So, the action of the quasigroup $(A_n,\diamond,L,R)$ on $n$-sized words with distinct letters where $n\geq 1$ allows the construction of any element in $M(n)$.
\end{proof}
\begin{dfn}
 The elements in $M(n)=(\mathbb{M}_n(A_n),\Ss_n)$ are called {\it symmetric sequences} of length $n$; whereas sequences of length $n$ defined from the alphabet $A_n$ and carrying an action of a quasigroup on $A_n$ are {\it translated sequences} of length $n$ and denoted  $M_{QG}(n)$
  \end{dfn}

Their relation can be described in the next diagram:
\begin{center}\begin{tikzcd}
M(n)\arrow[hookrightarrow,"i"]{r} & M_{QG}(n)\arrow[hookrightarrow,"j"]{r} \arrow[d,"f_n"] & M(A_n)\arrow[dl,"g"]\\
 &(A_n,\circ)& 
 \end{tikzcd}\end{center}

In short, we can consider two different objects, one being contained in the other one.  
The first one is given by the collection of symmetric sequences $\{M(n)\}_{n\geq 1}$. Morphisms between elements of $M(n)$ are given by the right action of the symmetric group $\Ss_n$. The second object is given by the collection $\{M(A_n)\}_{n\geq 1}$ where morphisms are translation maps, generated by $L$ and $R$ acting componentwise on the $n$-tuples forming $n$-sized words of $M(A_n)$. 

can be obtained from translated sequences. All these structures can be obtained set theoretically from the free magma. The free magma allows a decomposition of the magma $(A_n,\circ)$ by word length. 

So, if we restrict our considerations to the case, where translation maps form permutations, the following diagram appears:
\begin{center}\begin{tikzcd}
\arrow[ddd,"\cong"]  \mathbb{M}_n(A_n) \arrow[dr] \arrow[hookrightarrow,"i"]{r}&  M_{n}(A_n)\arrow[d,"f_n"]   \arrow[hookrightarrow,"j"]{r}&\arrow[ld,"g"]  M(A_n)\arrow[ddd,"\cong"]  \\
  &    (A_n,\circ)    \arrow[d,"\sigma\in\Ss_n"]        &  \\
 &  (A_n^\sigma,\circ)      & \\
\mathbb{M}_n(A_n^\sigma) \arrow[hookrightarrow,"i"]{r}\arrow[ur,"f_n^\sigma"]   &   \arrow[u]  M_{n}(A_n^\sigma)  \arrow[hookrightarrow,"j"]{r}      & M(A_n^\sigma)   \arrow[lu]     \arrow[uuu]  
\end{tikzcd}\end{center}

where we have the inclusion morphisms $i,j$. The inclusion $j$ goes from the degree $n$ component generated by all words of length $n$ to the free magma; $A_n^\sigma=\sigma(A_n)$, and $\sigma\in \Ss_n$ is a permutation obtained from and right combination of translation maps $L$ and $R$ such that they satisfy condition \eqref{E:equation} for any pair of transpositions.

\begin{lem}
The composition of translation maps acting on a set $A_n$ is associative.  
 \end{lem}
 For a word of length 3 one can for instance take: 
Given the following data:
\begin{itemize}
\item Word: $(x_1,x_2,x_3)$ 
\item $f=((L_a,R_b,L_c),(x_1,x_2,x_3))$
\item $g=((R_a,R_c,L_b),(y_1,y_2,y_3))$
\item $h=((L_b,L_c,R_b),(z_1,z_2,z_3))$
\end{itemize}

\[LHS=((h\circ g)\circ f)(x) =(L_b(R_a(y_1)),L_c(R_c(y_2)),R_b(L_b(y_3)))\circ f(x)=\]\[
(L_b(R_a(L_a(x_1))), L_c(R_c(R_b(x_2))),R_b(L_b(L_c(x_3)))) \]

\[RHS=(h\circ (g\circ f)(x)=h\circ (R_a(L_a(x_1)),R_c(R_b(x_2)),L_b(L_c(x_3)))=\]
\[L_b(R_a(L_a(x_1)),L_c(R_c(R_b(x_2))),R_b(L_b(L_c(x_3)))\]
So, RHS=LHS. 
 \begin{proof}
 One can easily check that for a word of length $n$ the statement is true using induction. 
 \end{proof}

\begin{lem}
Consider a (possibly non reduced) Latin square associated to a quasigroup $(A_n,\diamond,L,R)$, such that the first row (resp. column) corresponds to the sequence of the letters of a word $w\in\mathbb{M}_n(A_n)$. Then the $n-1$ other rows (resp. columns) of the Latin square form sequences of $n$-words being a permutation of $w$.
\end{lem}
\begin{proof}
Take an $n$-word with $n$ distinct letters, $w\in \mathbb{M}_n(A_n)$. We use (a possibly non reduced version of) the Latin square, such that the sequence of letters in the word $w$ forms the first row or first column. The multiplication table, forming the latin square, gives permutations of the word $w$. Multiplying each letter of the word $w$ by an element $a\in A_n$ gives a new row or column indexed by the element $a$.
\end{proof}

 \begin{ex}
 Using the previously discussed example, it is easy to check the above lemma by taking on the first row the monomial $(badc)$. It gives a non reduced Latin square, described below. 
 \begin{center} 
\renewcommand\arraystretch{1.3}
\setlength\doublerulesep{0pt}
\begin{tabular}{r||*{4}{2|}}
$\cdot$ & b & a & d & c  \\
\hline\hline
$b$ & d & c & b & a \\ 
\hline
 $a$ & c & b & a & d \\ 
\hline
$ d$ & b & a & d & c \\ 
\hline
 $c$ & a & d & c & b \\ 
\hline
\end{tabular}
\end{center}
It is easy to see that applying the translation maps to the entire word $w$ gives three other permutations: $cbad$, for $L_a$; $dcba$ for $L_b$ and $adcb$ for $L_c$. 
 \end{ex}
 \section{Modified  parenthesised Braids as a key to code-correction}\label{S:mPaB}
 In this section, we introduced an object that we call the {\it modified  parenthesised Braids}. This object serves as a model to consider the space of all paths of errors that may occur during the transmission of a given information. A particular advantage of this object is that it helps visualise the distorsion process easily (via modified braids) and thus leads to more facility for the correction process.

Previously, we have shown that for all $n\geq1$, each $M_{n}(A)$ comes equipped with a decomposition:
\[M_{n}(A)=\sqcup_{1\leq p\leq n-1} M_{p}(A)\times M_{n-p}(A),\]
indexed by the partitions of the integer $n$. 

This precise procedure allows to define the parenthesis in the case of parenthesised words. The number of ways to insert $n$ pairs of parentheses in a word of $n+1$ letters is the celebrated Catalan number $Cat(n)$. For $n=2$ there are 2 ways: $((ab)c)$ or $(a(bc))$, as for $n=3$ there are 5 ways: $((ab)(cd)), (((ab)c)d), ((a(bc))d),$ $(a((bc)d)), (a(b(cd)))$.
This is in bijection with all the binomial$(2n,n)$ paths on $\Z$ lattice that start at $(0, 0)$ end at $(2n, 0)$,
where each step corresponds either to making a $(+1,+1)$ step or a $(+1,-1)$ step. The number of such paths that never go below the $x$-axis (also known as the Dyck paths) is $C(n)$. 

\begin{prop}
Let $n\geq 1$ be an integer. 
To every  parenthesised word in $w\in M_{n}(A_n)$ one can construct a corresponding Dyck path of size $n$ in the real plane starting at $(0, 0)$ ending at $(2n, 0)$.
\end{prop}
\begin{proof}
Using the inductive construction on $M_{n}(A_n)$ mentioned above, it is easy to obtain a word with parenthesis. Now, concerning the Dyck paths, a step up i.e. with coordinates $(+1,+1)$ correspond to an opening of  parenthesis and a step down (i.e. step with $(+1,-1)$ coordinates) corresponds to a closing parenthesis. Some vertices of the Dyck path may carry a label which corresponds to the letter(s) of the corresponding block in the word. \end{proof}
Note that some vertices can be labeled by an $r$-tuple accordingly to the corresponding block.

The following Dyck path corresponds to the word $((abc)d)(ef)(g(hi))$
\begin{center}\includegraphics[scale=0.39]{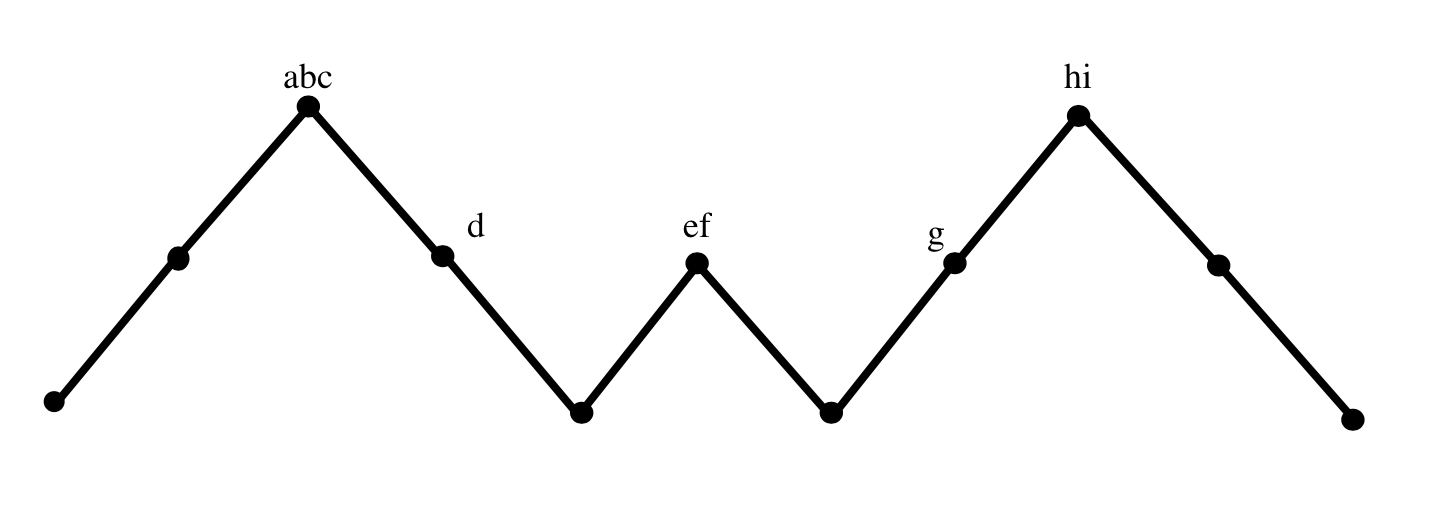}.
\end{center}

It is important to distinguish $\cPT$ from the actual quasigroup providing the left and right translation maps. 

\begin{prop}
The parenthesised translated words $\cPT$ form a category.
\begin{itemize}
\item Objects are the collection $M_n(A_n)$ of parenthesised words of length $n$, where $n\geq 1$ formed from the alphabet $A_n$.  
\item Morphisms are non-empty for words of the same length. These morphisms are given by the componentwise action on the $n$-tuple of letters  given by the translation maps $L$ and $R$ applied to the letters of the words. Translation maps can permute letters of a word or shift a given letter to another one and obey to the multiplication table given by the Latin square.
\end{itemize}
\end{prop}
\begin{proof}
Objects are the collection of parenthesised words of size $n$, where $n\geq 0$, formed from a $Card(n)$ alphabet $A_n$ and where letters are allowed to repeat. Morphisms map an $n$-sized word to another $n$-sized word, using the left and right translation maps. Those maps act accordingly to the corresponding $n\times n$ Latin square, on the letters of the word. Composition of the translation maps are clearly allowed.  There exists an identity morphism $Id$ so that given a word $w$ we have that $Id_w:w\to w$. This is possible since for any $L_r(x):x\mapsto rx=c$ (for $x \in A_n$ a letter of $w$) there exists a divisor of $c$ giving back $x$ (by definition of a quasigroup). The associativity holds. Consider a sequence $(x_1,\cdots, x_n)$. Then one can act on each letter using the left and right translation maps so that $(x_1,\cdots, x_n)$ is mapped to  $(T_1(x_1),\cdots, T_n(x_n))$ where $x_1,\cdots, x_n\in A_n$. $T_i$ are the translation maps which can be obtained as a composition of right and left maps acting on each letter independently. The composition of translation maps is associative. In relation to this it is important to distinguish the composition of translation maps operation acting on a sequence of letters and the multiplication operation on the quasigroup. 
\end{proof}

\medskip 

\subsection{Parenthesised modified braid groupoid $\mPaB$}
In this section, we rely on the construction presented in~\cite{BarN,BrHoRo,Dri90}, in order to show that the same structure applies for the modified  parenthesised braids as for the parenthesised braids.

Roughly speaking by parenthesised braid we mean a braid whose ends (i.e. top and bottom lines) correspond to parenthesised ordered points along a line. Let $B$ be such a parenthesised braid with $n$ strands.
\smallskip 

 \smallskip 

In other words, the operad of parenthesised braids $\PaB$ is the operad in groupoids
defined as follows (see Def. 6.11~\cite{BrHoRo}).
\begin{itemize}
\item  The operad of objects is the magma operad, i.e. $Ob(\PaB) = M=\{M(n)\}_{n\geq 0}$.
\item  For each $n \geq 0$, the morphisms of the groupoid $\PaB(n)$ are morphisms in
the (colored) braid groupoid, denoted $\CoB(n)=\{\CoB(n)\}_{n \geq 0}$, where
 $Hom_{\PaB(n)}(p,q)=Hom_{\CoB(n)}(u(p),u(q))$ with $p,q\in \Ss_n$ and the morphisms are braids associated to the permutation $qp^{-1}$. 
\end{itemize}

For the reader's convenience we recall the definition of the collection of groupoids $\CoB = \{\CoB(n)\}_{n \geq 0}$, following Def. 6.1~\cite{BrHoRo}: 
\begin{itemize}
 \item $\CoB(0)$ is the empty groupoid.
 \item For $n >0$, the set of objects $Ob(\CoB(n))$ is $\Ss_n$. For our own purposes, we propose to modify here the classical perspective on this object by defining the generators of $\Ss_n$ from the point of view of translation maps, i.e. given by some specific combinations of translation maps, lying in the space of all translation maps denoted $T_n$. 
\item A morphism in $\CoB(n)$ from $p$ to $q$ is a braid $\alpha \in B(n)$ whose associated
permutation is $qp^{-1}$.
\end{itemize}
The categorical composition in $\CoB(n)$ is given by the concatenation operation of braids
\[Hom_{\CoB(n)}(p, q) \times Hom_{\CoB(n)}(q, t) \to Hom_{\CoB(n)}(p, t)\]
inherited from the braid group. We write $a \cdot b$ for the categorical composition of $a$ and $b$.
%%%%HELLO 

\smallskip 
\begin{rem}
As one can see, this type of object fits the description of errors of type (1) discussed earlier. Those errors are mainly given by a permutation of letters in a word. However, this forms a very restrictive subclass of possible mistakes. Moreover, it is rare to form sentences of words having all letters distinct. Therefore, we add the class (2) of possible errors to our investigations and so it is necessary to modify the definition of $\PaB$ slightly so as to obtain a larger panel of possible errors. 
\end{rem}

The above definitions being settled, we introduce the notion of modified parenthesised braid $\mPaB$. It is reminiscent to its original version $\PaB$, in the sense that objects are given by parenthesised words. Somehow,  since letters are allowed to repeat in a word, one needs to introduce a modification of the braid. This modification of the braid is given by introducing two supplementary operations: the {\it pinching} operation and the {\it attaching} operation. These operations are a geometric representation of the left (resp. right) translation of a letter into another one, if this letters has already been used in the word.  

\begin{dfn}
Consider a pair of strands in a given braid. We say that there exists a {\it pinching operation} whenever those two strands are pinched together at a point. This point lies neither on the top nor on the bottom lines of the braid.
We say that there exists an {\it attaching operation} if the pinching lies on either the the top or bottom lines of the braid. 
 \end{dfn}
 
 \begin{figure}
  \begin{center}
  \includegraphics[scale=0.15]{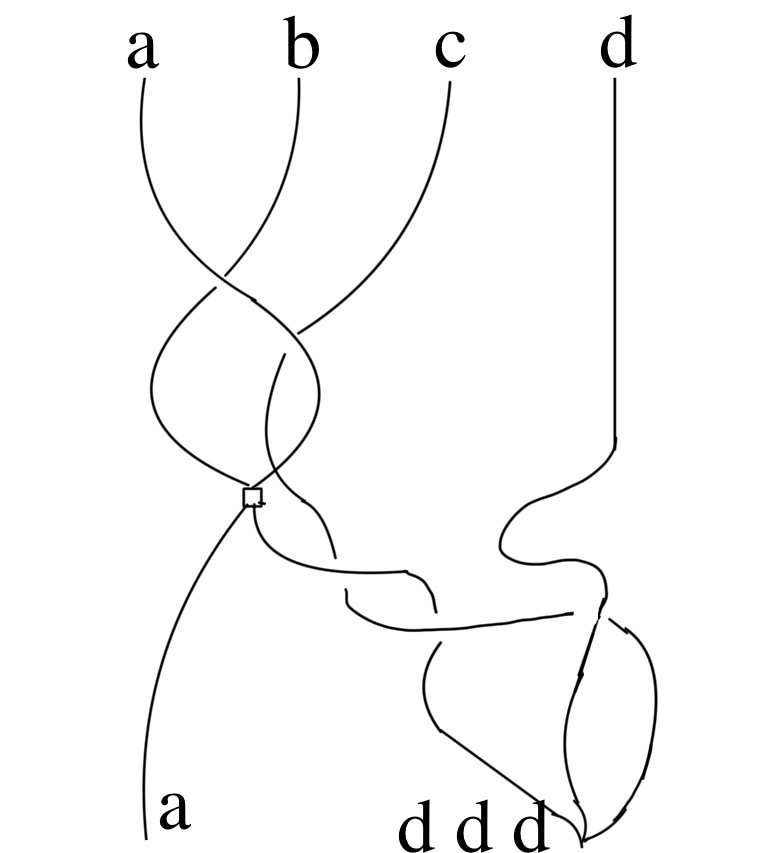}\end{center}
  \caption{Pinching and attaching points}\label{F:pinchA}
  \end{figure}
  
See an illustration of this in Fig. \ref{F:pinchA}, where a pinching point is presented between the strands starting at $a$ and $b$ and an attaching point is presented for the strands starting at $b,c,d$.  The attaching operation, occurs during the transformation of the word $abcd$ into the word $addd$.

\begin{rem}
Note that the pinching/ attaching does not imply that the strands have been intertwined. An intertwining of two strands amounts to solving equation \ref{E:equation} for a pair of translation maps applied to a pair of letters.
\end{rem}

 \begin{ex}
 We provide an additional example of a pinched (modified) braid in Fig. \ref{F:pinch}. As one can see, we have a pair of parenthesised words $(ab)(cd)$ on the top line and $a(b(cd))$ on the bottom line. The pinching occurs after that $c,d$ and $a,b$ are swapped and the word morphism maps $(ab)(cd)\to (ba)(dc)$.
The following translation map applied to $c$ (given by $L_x(c)=d$) gives the new word $b(a(dd))$, where two strands are pinched at $d$.  
 \begin{figure}  \begin{center}\includegraphics[scale=0.35]{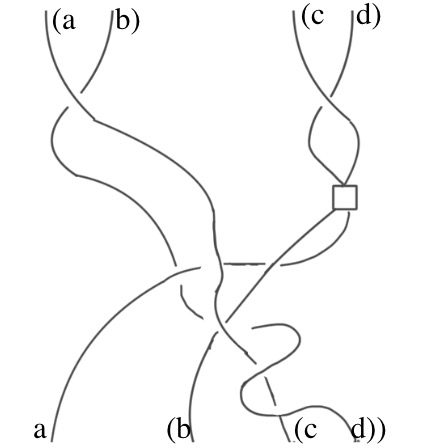}\end{center}\caption{One pinching point}\label{F:pinch}
  \end{figure}
 \end{ex}
 
 We are interested now in considering the errors of type (4),(5) and (6). This implies adding letters, losing letters or even duplicating letters. 
First, we introduce the modified parenthesised braids. 

  \begin{dfn}
 Let $\mPaB$ be the category whose objects are parenthesised words and morphisms are given by a pair $(P_{\star},\sum_{j=1}^k\beta_j\tilde{B}_{j})$, constituted from: 
 \begin{itemize}
 \item a morphism in the category of parenthesised translations, denoted $P_{\star}$;
 \item the linear sum of parenthesised modified braids $\sum_{j=1}^k\beta_j\tilde{B}_{j}$ is defined such that the skeleton of $\tilde{B}_{j}$ is $P_{\star}$; the coefficients $\beta_i$ lie in some "ground algebra". 
 \end{itemize}
 The composition law in $\mPaB$ is given by the bilinear extension of the composition law of modified parenthesised braids. 
\end{dfn}

  \begin{prop}
  The modified parenthesised braids $\mPaB$ forms a groupoid where:
 
\begin{itemize}
\item  The objects are the collections $n$-sized parenthesised translated words.
\item  For each $n \geq 0$, the morphisms are in
 the modified braid groupoid $\mCoB(n)$, where: $Hom_{\mPaB(n)}(p_{\star},q_{\star})=Hom_{\mCoB(n)}(u(p_{\star}),u(q_{\star})),$
with $p_{\star},q_{\star}\in T_n$ being translations. \end{itemize}

 The symbol $\mCoB = \{\mCoB(n)\}_{n \geq 0}$ defines the modified (colored) braids which consists of a collection of groupoids $\mCoB(n)$ defined as follows.
 \begin{itemize}
 \item $\mCoB(0)$ is the empty set.
 \item For $n >0$, the set of objects $Ob(\mCoB(n))$ are the translations $T_n$ (containing $\Ss_n$) on a set of $n$ elements, where rules of translating elements are given by the corresponding $n\times n$ Latin squares.
\item  A morphism in $\mCoB(n)$ from the translation $p_{\star}$ to the translation $q_{\star}$ is a modified braid $\alpha_{\star}\in mB(n)$ whose associated
translation is $q_{\star}p_{\star}^{-1}$.
\end{itemize}
\end{prop}
\begin{proof}

By definition, a groupoid is a small category in which every morphism is an isomorphism (i.e. it is invertible). A  groupoid  is given by a set 
of objects; here we take the collection of $n$-sized parenthesised words. 

For each pair of objects $w$ and $w'$ in the set of parenthesised words, there exists a (possibly empty) set of morphisms from $w$ to $w'$. 
Here those morphisms are given by translating one $n-$word into another one by using the translation maps $L$ and $R$.  This morphism can amount to a permutation of the letters of the words (like in the classical $\PaB$ case) but does not have to. In particular, letters can be shifted into other letters, creating thus a word where letters repeat.  

Now, for every word $w$, there is a designated element $\mathrm{id}_w$. This is due to the fact that in a quasi group every element is invertible and that we can, in addition, add the notion of neutral element giving the identity (and forming thus a loop). 

For each triple of objects $w, w'$, and $w''$, one has a composition of translation maps allowing the morphism $f:w\to w'$ to be composed with $g:w'\to w''$ and giving $gf: w\to w''$. Furthermore the morphisms are invertible (and this follows from the definition of translation maps).
This construction then leads to the modified braids. For a given translation $p$ and $q$ of a word, one defines an associated modified braid, whose translation is $qp^{-1}$. The domain and range are the parenthesised words corresponding to the permutations $p$ and $q$ respectively. 

\end{proof}

In order to construct a ``tower'' of modified parenthesised braids, the key setup is already existent  for the category $\PaB$ where one has the {\it extension operations, cabling operations, strand removal operations}. The whole point of the next proposition and lemma will be to first define rigorously the modified parenthesised braid groupoid and to show that the extension, cabling and strand removal operations can be inherited on this new object.

\begin{lem}
Let $(\PaB,d_i,s_i,d_0)$  be the category of parenthesised braids, equipped with the three operations known as: extension operation $d_0$, cabling operation $d_i$ and strand removal $s_i$ operation. Then, those operations are inherited on the category $\mPaB$. 
\end{lem}
\begin{proof}
Consider $B$ a parenthesised braid with $n$ strands.
   \begin{itemize}
   \item {\it Extension operations}. Given a braid $B$, one adds on the left-most (or right-most) side a straight strand, with ends regarded as outer-most. This operation of adding a straight strand does not encounter any obstruction for the modified braids and so it is inherited from $\PaB$. 
\item {\it Cabling operations}. For $1 \leq i \leq n$, let us consider the parenthesised braid obtained from $B$ by doubling its $i$-th strand (counting at the bottom). This cabling operation can be applied in any situation: either when strands are separated as in the classical braid setting or when they are attached/ pinched. So again this operation is inherited from $\PaB$.
\item {\it Strand removal operations}. For $1 \leq i \leq n$ be the parenthesised braid
obtained from $B$ by removing its $i$-th strand (counting at the bottom). Removing a strand in the modified braid holds also in this case. \end{itemize}
So, all three operations are well defined for $\mPaB$. 
\end{proof}

\smallskip 

We will now prove that  $\PaB$  is a full subcategory of $\mPaB$.
\begin{cor}
The category $\PaB$ is a full subcategory of $\mPaB$ i.e. there exists a full inclusion:
\[ \PaB\hookrightarrow \mPaB.\]
\end{cor}
\begin{proof} 
To have a full subcategory of $\mPaB$ it is necessary that for any objects $x,y$ in $\PaB$ every morphism of in $\mPaB$ is also in $\PaB$. This is a true statement since any permutation is given by translations maps $L$ and $R$ (satisfying conditions \eqref{E:equation} for any transposition). 
\end{proof}

In the setting of the category $\PaB$, there exists a functor $\bS$ called the {\it skeleton functor} on whose image is the category of  parenthesised permutations $\cPaP$. The operations $d_i, s_i$ of cabling and strand removal operations are naturally defined on $PaP$. The skeleton functor $\bS$ intertwines the $d_i$'s and the $s_i$'s acting on  parenthesised braids and on parentesized permutations. 
The same type of object exists for $\mPaB$ and parenthesised translations. 

Denote the  parenthesised translations $\cPT$. The skeleton functor $\bS_{\cPT}$ for the $\mPaB$ is the identity on objects, where  objects are  parenthesised words in  $\cPT$ i.e words equipped with parenthesis and where letters can be permuted or translated, giving thus possibly  parenthesised words with non distinct letters.   

\begin{prop}
The category $\cPT$ together with the functor $\bS_{\cPT}:\mPaB\to \cPT$ is a fibered linear category. 
\end{prop}
\begin{proof}
The category $\cPT$ together with the functor $\bS_{\cPT}:\mPaB\to \cPT$ forms a fibered linear category, for the following reasons.
First, $\cPT$ has the same objects as $\mPaB$ and the skeleton functor is the identity on objects. Secondly, the inverse image $\bS_{\cPT}^{-1}(P_{\star})$ of every morphism $P_{\star}\in \cPT$ is a linear composition of left and right translation maps $L$ (and $R$) and, similarly to the case of the parenthesised braids, it forms a linear space. The composition maps in $\mPaB$ are also bilinear in the natural sense. 
\end{proof}

\subsection{The Grothendieck--Teichm\"uller group and modified parenthesised braids}\label{S:GT}

We now discuss the following theorem.  

\begin{thm}\label{T:GT}
The pro-unipotent Grothendieck--Teichm\"uller group is contained in the groups of structure preserving automorphisms $Aut(\widehat{\mPaB})$.
\end{thm}

In order to prove this statement, we apply the method on fibered linear categories such as shown in detail in Sec.2.1.1. ~\cite{BarN}).
Consider $\mPaB$ and $\cPT$, being respectively the categories of modified parenthesised braids and parenthesised translations.  Define a subcategory of the fibered linear category $(\mPaB, \bS : \mPaB \to\cPT)$ as follows. Let $P_{\star}$ be a morphism in $\cPT$. Choose {\bf a linear subspace} in each $\bS^{-1}(P_{\star})$, so that the system of subspaces chosen is closed under composition. Closed under composition means that two modified braids (such that the bottom line of the first modified braid coincides with the top line of the second one) lie both in the linear subspace generated by $\bS^{-1}(P_{\star})$ and defines another modified braid belonging to $\bS^{-1}(P_{\star})$. We construct an ideal $\bI$ in $(\mPaB, \bS : \mPaB \to\cPT)$, which is a subcategory. The quotient $\mPaB/\bI$ of the fibered linear category $\mPaB$ by the ideal $\bI$ is again a fibered linear category.

These fibered linear categories are compatible with further operations allowing the construction of the inverse limit of an inverse system of fibered linear categories (fibered in a compatible way over the same category of skeletons). So, if $\bI$ is an ideal in a fibered linear category $\bB$, one can form the $\bI$-adic completion and this $\bI$-adic completion is again a filtered fibered linear category.

\smallskip 

\begin{lem}
There exists a tower of modified parenthesised braids \break
$(\widehat{\mPaB}, \widehat{\mPaB}\to \cPT, d_i, s_i)$, where 
$\widehat{\mPaB}$ is the unipotent completion of $\mPaB$.   
\end{lem}
\begin{proof}
Define the subcategory of the fibered linear category $(\mPaB, \bS_{\cPT} : \mPaB \to \cPT)$ as follows. Let $P_{\star}$ be a morphism in $\cPT$. We choose a linear subspace in each $\bS^{-1}(P_\star)$, so that the system of subspaces chosen is closed under composition
(two translations such that the range of the first translation $T_1$ is the domain of the second $T_2$ and both lying in the linear subspace in $\bS^{-1}(P_{\star})$ define a translation $T_1\circ T_2$ also in $\bS^{-1}(P_{\star})$).

As mentioned earlier, an ideal in $(\mPaB, \bS_{\cPT}: \mPaB \to \cPT)$ is a subcategory $\bI$ so that if at least one of the two composable morphisms $T_1$ and $T_2$ in $\cPT$ is actually in $\bI$, then their composition $T_1 \circ T_2$ is also in $\bI$. The ideal  $\bI^m$ is such that morphisms of $\bI^m$ are all those morphisms in $\mPaB$ that can be presented as compositions of $m$ morphisms in $\bI$.

In particular, given that $\bI$ is an ideal of a fibered linear category $\mPaB$, one can form the $\bI$-adic completion $\widehat{\mPaB} = \lim_{m\to \infty}\mPaB/\bI^m$, where the $\bI$-adic completion is a filtered fibered linear category.

Take $\bI$ to be the augmentation ideal of $\mPaB$ formed from all pairs $(P,\beta_jT_j)$ in which $\sum \beta_j=0$. Powers of this ideal defines the {\it unipotent filtration} of $\mPaB$, which is denoted $\cF_\mPaB=\bI^{m+1}$.  

Let $\mPaB^{(m)} =\mPaB/\cF_m \mPaB = \mPaB/\bI^{m+1}$ be the $m$-th unipotent quotient of $\mPaB$, and let 
$\widehat{\mPaB}= \lim_{ m\to \infty} \mPaB^{(m)}$,
be the unipotent completion of $\mPaB$. The fibered linear categories inherit the operations $d_i$ and $s_i$ and a coproduct and filtration $\cF_{*}$.
\end{proof}

\smallskip 

Finally  this construction leads to considering the automorphism group of the tower of modified braids $Aut(\widehat{\mPaB})$,
being the group of all functors $\widehat{\mPaB}\to \widehat{\mPaB}$, covering the skeleton functor, intertwining $d_i, s_i,\square$ (the coproduct functor $\square:\mPaB\to\mPaB\otimes \mPaB$) and fixing the elementary braid $\sigma$ (a crossing of two strands). 

\begin{proof}[Proof\, of\, Theorem~\ref{T:GT}]
We have shown that $\widehat{\mPaB}$ is an enriched version of the construction of $\widehat{PaB}$ in \cite{BarN} (in the sense that it inherits its properties and operations but has some additional structures). As well we obtained that $\PaB$ is a subcategory of $\mPaB$. 

By definition, we have that $\widehat{GT}=Aut(\widehat{\PaB})$, where $Aut(\widehat{\PaB})$ is the group of all functors $\widehat{\PaB}\to \widehat{\PaB}$ that cover the skeleton functor, intertwine $d_i, s_i$ and $\square$ and fixes $\sigma$.
The inclusion of $\PaB$ in $\mPaB$ implies that $Aut(\widehat{\PaB})$ is included in $Aut(\widehat{\mPaB})$.
So, the pro-unipotent Grothendieck--Teichm\"uller group is contained in the groups of structure preserving automorphisms $Aut(\widehat{\mPaB})$.
\end{proof}

Finally, using the inclusion theorem of~\cite{Brown} stating that the motivic Galois group is included in $Aut(\widehat{\PaB})$, we can conclude that: 
\begin{cor}\label{C:mot}
The motivic Galois group is contained in the automorphism group $Aut(\widehat{\mPaB})$.
\end{cor}
We can interpret $\widehat{\mPaB}$ as modelling a situation where one considers infinitely many errors occurring. The information that it tells us is that the in a pro-unipotent completion of the space of ways of making errors has among others the behaviour of the motivic Galois group encapsulated within it.

\subsection{Conjectures and open questions}\label{S:conj}
Moufang loops turn out to be central in geometry of information, in particular for statistical manifolds (related to exponential families) and codes/error-codes.  For instance, symmetries of spaces of probability distributions, endowed with their canonical Riemannian metric of information geometry, have the structure of a commutative Moufang loop.

In a different setting, there exists a connection between Moufang loops algebraic geometry. Recall from \cite{Cu} the relation between Moufang loops and the set of algebraic points of a smooth cubic curve in a projective plane $\Pp^2_K$ over a field $K$. The set $E$ of $K$-points of such a curve $X$ forms a $CML$ with composition law  $x\circ y = u\star (x\star y)$, if $ u + x + y$ is the intersection cycle of $X$ with a projective line $\Pp ^1_K \subset \Pp^2_K$.

\smallskip 

Given that symmetries of statistical manifolds have the structure of $CML$ and that Manin's conjecture (coming from algebraic geometry) is true in the case of statistical manifolds, it leads to think that there is a stronger connection between the $CML$ coming from algebraic geometry and the $CML$ in the statistical manifolds. So, an intriguing question following from the properties of statistical manifolds defined above would be to determine whether the set $E$ of $K$-points of a pre-Frobenius statistical manifold has the structure of a $CML$. 

\bibliographystyle{alpha}

\bibliography{Bday}

\end{document}